\documentclass[a4paper]{amsart}
\usepackage{enumitem}
\usepackage{calc}

\usepackage{lmodern} %Type1-font for non-english texts and characters

\usepackage{graphicx} %%For loading graphic files

\usepackage{amsmath}
\usepackage{amssymb,tikz}
\usepackage{pict2e}
\usepackage{amsthm}
\usepackage{amscd}
\usepackage{mathrsfs}
\usepackage{todonotes}
\usetikzlibrary{calc} % This is to add two points in tikzpictures

\usepackage{yfonts}

\usepackage{marvosym}
\usepackage[percent]{overpic}

\newtheorem{theorem}{Theorem} [section]
\newtheorem{proposition}[theorem]{Proposition}

\newtheorem{example}{Example}[section]
\newtheorem{assumptions}[theorem]{Assumptions}

\newtheorem{remark}[theorem]{Remark}
\theoremstyle{definition}
\newtheorem{definition}[theorem]{Definition}

\newcommand{\C}{\mathbb{C}}

\newcommand{\Tr}{{\rm Tr}}

\newcommand{\supp}{{\rm supp}}
\newcommand{\dx}{ {\rm{d}} }

\def\XXint#1#2#3{{\setbox0=\hbox{$#1{#2#3}{\int}$}
		\vcenter{\hbox{$#2#3$}}\kern-.5\wd0}}

\usepackage{tikz}
\usetikzlibrary{arrows}
\usetikzlibrary{decorations.pathmorphing}
\usetikzlibrary{decorations.markings}
\usetikzlibrary{patterns}
\usetikzlibrary{automata}
\usetikzlibrary{positioning}
\usepackage{tikz-cd}
\tikzset{->-/.style={decoration={
			markings,
			mark=at position #1 with {\arrow{latex}}},postaction={decorate}}}

\tikzset{-<-/.style={decoration={
			markings,
			mark=at position #1 with {\arrowreversed{latex}}},postaction={decorate}}}

\usetikzlibrary{shapes.misc}\tikzset{cross/.style={cross out, draw, 
		minimum size=2*(#1-\pgflinewidth), 
		inner sep=0pt, outer sep=0pt}}

\usepackage{pgfplots}
\numberwithin{equation}{section}

\usepackage[colorlinks=true]{hyperref}
\hypersetup{urlcolor=blue, citecolor=red, linkcolor=blue}

\begin{document}
	\title{Determinantal point processes conditioned on randomly incomplete configurations}
	\author{Tom Claeys and Gabriel Glesner}
	\maketitle
	\begin{abstract}
		For a broad class of point processes, including determinantal point processes, we construct associated marked and conditional ensembles, which allow to study a random configuration in the point process, based on information about a randomly incomplete part of the configuration.
		We show that our construction yields a well behaving transformation of sufficiently regular point processes.
		In the case of determinantal point processes, we explain that special cases of the conditional ensembles already appear implicitly in the literature, namely in the study of unitary invariant random matrix ensembles, in the Its-Izergin-Korepin-Slavnov method to analyze Fredholm determinants, and in the study of number rigidity.
		As applications of our construction, we show that a class of determinantal point processes induced by orthogonal projection operators, including the sine, Airy, and Bessel point processes, satisfies a strengthened notion of number rigidity, and we give a probabilistic interpretation of the 
		Its-Izergin-Korepin-Slavnov method.\\ \medskip 
		
		Pour une large classe de processus de points, contenant les processus de points d\'eterminantaux, nous construisons les ensembles marqu\'es et conditionn\'es associ\'es qui permettent d'\'etudier une configuration al\'eatoire dans le processus de points, \`a partir d'information sur une partie al\'eatoirement incompl\`ete de cette configuration. Nous montrons que notre construction produit une transformation qui se comporte bien pour des processus de points suffisamment r\'eguliers. Dans le cas des processus de points déterminantaux, nous expliquons que certains cas particuliers de ces ensembles conditionn\'es sont d\'ej\`a apparus implicitement dans la litt\'erature, \`a savoir dans l'\'etude des ensembles de matrices al\'eatoires unitairement invariant, dans la m\'ethode de Its, Izergin, Korepin et Slavnov pour analyser les d\'eterminants de Fredholm, et dans l'\'etude de la rigidit\'e des nombres. Comme application de notre construction, nous montrons qu'une classe de processus de points d\'eterminantaux induits par des projections orthogonales, comprenant les processus de points sinus, Airy et Bessel, satisfait une propri\'et\'e plus forte que la rigidit\'e des nombres, et nous donnons une interpr\'etation probabiliste de la m\'ethode de Its, Izergin, Korepin et Slavnov.
		
	\end{abstract}

	\section{Introduction}\label{section:1}
	
	\subsection{Background and motivation}
	Determinantal point processes (DPPs) are point processes whose correlation functions can be written as determinants of a correlation kernel, and for which average multiplicative statistics are Fredholm determinants. Prominent examples of DPPs
	are the eigenvalue distributions of a large class of random matrix ensembles, distributions of particles in asymmetric exclusion processes and tiling models, distributions of non-intersecting random paths, and the zeros of Gaussian analytic functions. They are special cases of repulsive point processes, in which one can study relevant probabilistic quantities through the analysis of the correlation kernel and associated Fredholm determinants \cite{BaikDeiftSuidan, HoughKrishnapurPeresVirag, Johansson, Lyons, Macchi, Soshnikov}.
	
	\medskip
	
	A groundbreaking discovery 
	for the development of random matrix theory and more generally the study of DPPs
	has been the observation of Wigner and Dyson and their collaborators in the 1960s that energy levels of heavy nuclei can be accurately modelled by eigenvalues of random matrices. Despite his spectacular contributions, when Dyson looked back at his work on heavy nuclei in 2002 during the MSRI program {\em Recent Progress in Random Matrix Theory and Its Applications}, he explained \cite{Dyson} that the practical implications of his work on random matrices in nuclear physics were disappointing, because detectors were imperfect, and missing or spurious energy levels corrupted the data.
	Inspired by this, Dyson raised the question to develop error-correcting code for random matrix eigenvalues: given an imperfect observed spectrum of a random matrix, can one detect  missing or spurious eigenvalues?
	This would not be possible for point processes with independent points, because the positions
	of a fraction of the points in the process do not carry any information about the other points.
	In strongly correlated point configurations such as random matrix eigenvalues or DPPs, one can however hope to extract information based on incomplete data. According to \cite{Dyson}, Dyson did not suggest this direction of research because of its importance in nuclear physics, but purely because he believed it would lead to interesting mathematics.
{This question has been explored by Bohigas and Pato \cite{BohigasdeCarvalhoPato, BohigasPato} using randomly thinned random matrix eigenvalues, and has been picked up in the mathematics literature with the study of random thinnings of DPPs \cite{BerggrenDuits, Bothner, BothnerBuckingham, BothnerDeiftItsKrasovsky, BothnerItsProkhorov, Charlier1, Charlier2, Charlier3, CharlierClaeys, ForresterMays}, but a general mathematical theory for extracting information from the observation of randomly thinned DPPs has not been developed so far.}
	
	\medskip
	
	However, in the same spirit of attempting to extract information about DPPs from a partial observation,
	the remarkable property of number rigidity has recently been investigated.
	Informally, a point process is said to be number rigid if the configuration of points outside any bounded set determines almost surely the number of points inside the set. Important DPPs like the sine, Airy, and Bessel point processes arising in random matrix theory, are known to be number rigid \cite{Bufetov4, Ghosh, GhoshPeres, Lyons}, and in the case of the sine process, the distribution of the points inside a bounded set, conditioned on the configuration of points outside the set, has been studied and proved to converge to the sine process when the size of the interval grows \cite{KuijlaarsMinaDiaz}.
	
	\medskip
	
	In this work, for any sufficiently regular point process, and in particular for any DPP, we introduce a family of marked and conditional point processes which allow to formalize the following question: {\em Given a randomly incomplete sample of the point process, what can we say about the missing points?}
	Although these point processes have, to the best of our knowledge, not been introduced and studied on a general basis, special cases of them do already appear in the literature in various contexts, as we will explain in more detail later; firstly, unitarily invariant Hermitian random matrix ensembles are a special case of conditional ensembles associated to the Gaussian Unitary Ensemble (GUE); secondly, special cases of the conditional ensembles arise naturally in the Its-Izergin-Korepin-Slavnov (IIKS) \cite{IIKS} method to characterize Fredholm determinants via Riemann-Hilbert problems; and finally, special cases of the conditional ensembles have been studied in relation to number rigidity.

	\medskip
	{
Our objectives are:
\begin{enumerate}
\item to construct the marked and conditional ensembles rigorously;
\item to prove that the conditional ensembles define a well-behaving transformation which preserves the structure of DPPs and of several interesting subclasses of DPPs;
\item to introduce a refined notion of number ridigity and to show that important DPPs like the sine, Airy, and Bessel DPPs satisfy this notion of rigidity;
\item to illustrate that the IIKS method provides an effective framework to study the conditional ensembles via Riemann-Hilbert methods.
\end{enumerate}
	}
	\subsection{DPPs: generalities and main examples}
	
	Consider a measure space $(\Lambda,\mathcal B_\Lambda,\mu)$, with $\Lambda$ a complete separable metric space, $\mathcal B_\Lambda$ the Borel $\sigma$-algebra, and $\mu$ a locally\footnote{Here and for the rest of this paper, whenever we say that a property holds locally, we mean that it holds for any bounded Borel set.} finite positive Borel measure on $\Lambda$, i.e.\ satisfying \(\mu(B)<\infty\) for any bounded \(B\in\mathcal B_\Lambda\). 
	We will be mainly interested in $\Lambda=\mathbb R$ with the Lebesgue measure or $\Lambda$ the unit circle in the complex plane with the arc length measure, and the reader may prefer to keep only these examples in mind for the sake of simplicity.
	Let $\mathbb P$ be a simple point process on $\Lambda$, i.e.\ a probability measure on the set $\mathcal N(\Lambda)$ of locally finite point configurations in $\Lambda$ (see Section \ref{section:2} for a more precise definition of the probability space), such that there are a.s.\ no points with multiplicity $>1$.  We can represent such a configuration $\xi\in\mathcal N(\Lambda)$ as a locally finite counting measure
	\[\xi=\sum_{j\in J}\delta_{x_j},\]
	where $J$ is a countable index set, and $x_j\in \Lambda$, \(x_i\neq x_j\) when \(i\neq j\). 
	Recall (see e.g.\ \cite[Section 9.4]{DaleyVereJones}) that a simple point process on $\Lambda$ is characterized uniquely by its Laplace functional
	\[\mathcal L: B_+(\Lambda)\to \mathbb R^+:f\mapsto \mathcal L[f],\qquad \mathcal L[f]=\mathbb E e^{-\sum_{x\in\supp\,\xi}f(x)}=\mathbb E e^{-{\int_\Lambda f\dx \xi}},\]
	where $B_+(\Lambda)$ is the space of bounded non-negative measurable functions $f:\Lambda\to [0,+\infty)$ with bounded support.
	
	\medskip
	
	Some of our results hold for any sufficiently regular point process, but our main focus will be on DPPs, for which the correlation functions $\rho_k:\Lambda^k\to [0,+\infty)$ (see again Section \ref{section:2} for details) of all orders exist and can be written in terms of a correlation kernel $K(x_i,x_j)$ in determinantal form:
	\begin{equation}\label{correlationkernel}
		\rho_k(x_1,\ldots,x_k)=\det\left(K(x_i,x_j)\right)_{i,j=1}^k.
	\end{equation}
If $K:\Lambda^2\to\mathbb C$ is the kernel of a locally trace class operator $\mathrm K$ on $L^2(\Lambda,\mu)$, then the Laplace functional is a Fredholm determinant:
	\begin{equation}\label{eq:LaplaceFredholm0}
		\mathcal L[f]=\det\left(1-\mathrm M_{\sqrt{1-e^{-f}}}\mathrm K\mathrm M_{\sqrt{1-e^{-f}}}\right),\end{equation}
	with $\mathrm M_g$ the multiplication operator with $g\in L^\infty(\Lambda,\mu)$ on $L^2(\Lambda,\mu)$, and the determinant is given by Fredholm's formula 
	\begin{equation}\label{def:Fredholm}
		\det\left(1-\mathrm M_{\sqrt{g}}\mathrm K\mathrm M_{\sqrt g}\right)=\sum_{k=0}^\infty \frac{(-1)^k}{k!} \int_{\Lambda^k}\det\left({\sqrt{g(x_i)}}K(x_i,x_j)\sqrt{g(x_j)}\right)_{i,j=1}^k \prod_{j=1}^k\mathrm d\mu(x_j).
	\end{equation}
	Note that the kernel \(K\) might not be well defined on the diagonal of \(\Lambda^2\), however we can always assume that \(K(x,x)\) is chosen such that for any bounded Borel set \(B\) the following holds (see \cite{Soshnikov}):
	\begin{equation*}
		\mathrm{Tr}\, \left.\mathrm K\right|_{L^2(B,\mu)}=\int_BK(x,x)\mathrm d\mu(x).
	\end{equation*}
	For notational convenience, let us introduce a change of variable in the Laplace functional and define the {\em average multiplicative functional}
	\begin{equation}\label{eq:LaplaceFredholm}
		L[\phi]:=\mathbb E\prod_{x\in\supp\,\xi}(1-\phi)(x)=\mathcal L[-\log(1-\phi)],
	\end{equation}
	for $\phi:\Lambda\to[0,1]$ measurable and with bounded support, such that $L[\phi]=\det\left(1-\mathrm M_{\sqrt{\phi}} \mathrm K\mathrm M_{\sqrt{\phi}}\right)$ if $\mathbb P$ is the DPP with kernel of the operator $\mathrm K$.
	
	\medskip
	
	Besides DPPs, it will be insightful to keep in mind the example of a Poisson point process with bounded locally integrable intensity $\rho:\Lambda\to[0,+\infty)$, for which
	\begin{equation}\label{def:Poissoncorrelation}
		\rho_k(x_1,\ldots,x_k)=\prod_{j=1}^k\rho(x_j).
	\end{equation}

	\medskip

	In Sections \ref{section:3}--\ref{section:5}, we will consider some important subclasses of DPPs, which we already define now. 
	
	\begin{example}\label{Example-OPE} {\bf Orthogonal polynomial ensembles (OPEs).}
		Let $N$ be a positive integer and consider the point process consisting of configurations of $N$ real points $x_1,\ldots, x_N$ with joint probability distribution
		\begin{equation}\label{jpdf}
			\frac{1}{Z_N}\Delta(x_1,\ldots,x_N)^2\ \prod_{j=1}^N w(x_j)\mathrm dx_j,\qquad \Delta(x_1,\ldots, x_N)=\prod_{1\leq i<j\leq N}(x_j-x_i),
		\end{equation}
		where $Z_N$ is a normalization constant, and $w(x)$ is a non-negative integrable weight function decaying sufficiently fast as $x\to \pm\infty$, such that all the moments $\int_{\mathbb R}x^kw(x)dx$, $k\in\mathbb N$, exist.
		If $w(x)=e^{-2Nx^2}$, this is the distribution of the (re-scaled, such that the eigenvalues follow a semi-circle law on $[-1,1]$) eigenvalues of a random matrix from the GUE. If $w(x)=x^\alpha e^{-Nx}1_{(0,+\infty)}(x)$ with $\alpha>-1$, it is the distribution of the eigenvalues of a random matrix in the Laguerre-Wishart ensemble. More generally, if $w$ takes the form $w(x)=e^{-NV(x)}$ with $V$ real analytic and growing sufficiently fast at $\pm\infty$, 
		\eqref{jpdf} is the eigenvalue distribution of a random matrix in the unitary invariant ensemble 
		\[\frac{1}{\widehat Z_N} e^{-N{\rm Tr }V(M)} \mathrm dM,\]
		with $\dx M$ the Lebesgue measure on the space of $N\times N$ Hermitian matrices, and $\widehat Z_N$ a normalization constant.
		
		\medskip
		
		Similarly, let $N$ be a positive integer and consider the point process consisting of configurations of $N$ points 
		$e^{it_1},\ldots, e^{it_N}$ on the unit circle in the complex plane with joint probability distribution
		\begin{equation}\label{jpdfUC}
			\frac{1}{Z_N}\prod_{1\leq j<k\leq N}|\Delta(e^{it_1},\ldots,e^{it_N})|^2\ \prod_{j=1}^N w(e^{it_j})\mathrm dt_j,\qquad t_j\in[0,2\pi),
		\end{equation}
		where $Z_N$ is a normalization constant, and $w(e^{it})$ is a non-negative integrable weight function.
		If $w(e^{it})=1$, this is the distribution of the eigenvalues of a random matrix from the Circular Unitary Ensemble (CUE), or in other words a Haar distributed $N\times N$ unitary matrix. 
		\medskip
		
It is well-known that the above OPEs are DPPs, with correlation kernel $K_N$ built out of orthogonal polynomials on the real line or on the unit circle. We will study these ensembles in more detail in Section \ref{section:4}.
	\end{example}
	
	\begin{example}\label{example:projection} {\bf DPPs induced by orthogonal projection operators.}
		Consider a DPP with correlation kernel $K$
		whose associated integral operator $\mathrm K$ on $L^2(\Lambda,\mu)$, defined by
		\begin{equation}\label{integral operator}
			\mathrm K f(x)=\int_\Lambda K(x,y)f(y)\mathrm d\mu(y),
		\end{equation}
		is a locally trace class orthogonal projection onto a closed subspace $H$ of $L^2(\Lambda,\mu)$. As we will see, the OPEs from Example \ref{Example-OPE} are of this form, and the associated projection operators are then of rank $N$. We recall from \cite{Soshnikov} that a DPP defined by the kernel of a Hermitian locally trace class operator $\mathrm K$ has the property that the number of particles is a.s.\ equal to $N$, i.e. $\mathbb P\left(\xi(\Lambda)=N\right)=1$, if and only if $\mathrm K$ is a projection operator of rank $N$. 
We will also consider DPPs induced by infinite rank projection operators. Such DPPs arise for instance when taking scaling limits of the kernels $K_N$ from Example \ref{Example-OPE}: we mention the DPPs defined by the sine kernel, the Airy kernel, the edge Bessel kernel, and the bulk Bessel kernel \cite{Deift, Kuijlaars}. 
		More complicated kernels associated to Painlev\'e equations and hierarchies (see \cite{Duits} for an overview), arising as double scaling limits of OPEs, are also of this form. We will consider such DPPs and derive rigidity results for some of them in Section \ref{section:3}.
	\end{example}
	\begin{example}\label{example:integrable}
		{\bf DPPs with integrable kernels.}
		In line with the terminology of Its, Izergin, Korepin, and Slavnov \cite{IIKS}, we say that a kernel $K(x,y)$ is $k$-integrable if it can be written in the form
	{	\begin{equation}\label{k-intkernel}
				K(x,y)=\frac{\sum_{j=1}^k f_j(x)g_j(y)}{x-y}\qquad \mbox{with } \sum_{j=1}^k f_j(x)g_j(x)=0,
					\end{equation}
		for some functions $f_j, g_j:\Lambda\to\mathbb C$, $j=1,\ldots, k$.}
		The previous examples of OPEs on the real line and on the unit circle are $2$-integrable, and so are the sine point process, the Airy point process, and the Bessel point processes.
		
		There are however many DPPs with integrable kernels that are not induced by projection operators. Indeed, if a kernel $K(x,y)$ defines a DPP on $\Lambda$, then any kernel of the form $\phi(x)K(x,y)$ with $\phi:\Lambda\to [0,1]$ measurable also defines a DPP, namely the random thinning of the original DPP realized by removing each particle $x$ in the support of a random point configuration $\xi$ independently with probability $1-\phi(x)$ \cite{LavancierMollerRubak}. If $K(x,y)$ is of integrable form, it is easy to see that the same is true for $\phi(x)K(x,y)$, but even if $K(x,y)$ defines an orthogonal projection operator, $\phi(x)K(x,y)$ in general does not define a projection operator. DPPs with integrable kernels will be our topic of interest in Section \ref{section:5}. 
	\end{example}

	\medskip
	
	\subsection{Marking and conditioning: informal construction and statement of results}
	
	For any sufficiently regular point process $\mathbb P$, we can construct an associated marked point process in which we assign a random mark to each point independently. If the random mark is a Bernoulli random variable taking the value $0$ or $1$, 
	then the marked point process is a point process on $\Lambda\times\{0,1\}$, in which
	we interpret the points with mark $1$ as visible or observed particles, and the points with mark $0$ as invisible or unobserved particles.
	Concretely, we mark the points in the DPP by introducing a measurable marking function $\theta:\Lambda\to[0,1]$, and by assigning mark $1$ to particle $x$ in a configuration of the DPP with probability $\theta(x)$, and mark $0$ with probability $1-\theta(x)$. We denote the resulting marked point process as $\mathbb P^\theta$. The random marking splits a configuration $\xi$ on $\Lambda$ into configurations $\xi_0$ and $\xi_1$, where  $\xi_b$ is the configuration $\xi$ restricted to the points with mark $b$. We denote $\mathbb P^\theta_b$, $b=0,1$, for the marginal probability distribution of $\xi_b$, which is a random position-dependent thinning of the ground process $\mathbb P$.
	We will introduce these marked point processes in detail in Section \ref{section:2}, and gather some of their general properties in Proposition \ref{prop:marked}. 
	The point processes in which we are most interested here, are point processes obtained as conditional ensembles of this marked point process, by conditioning on the (observed) configuration of mark $1$ points. 

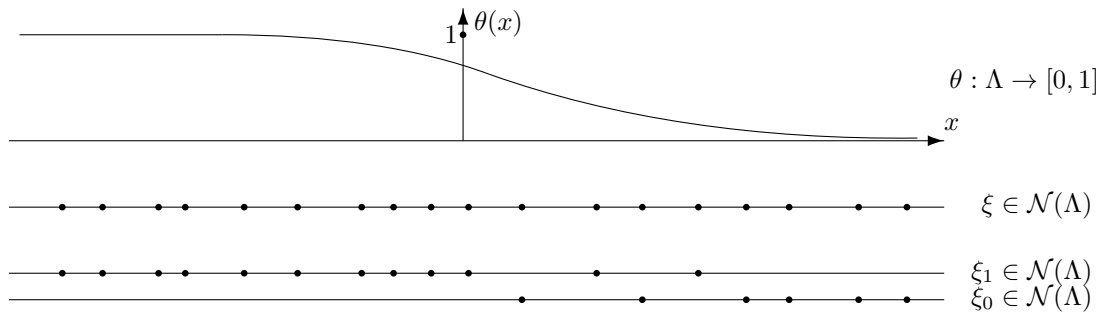
\begin{figure}[t]
\begin{center}
\begin{picture}(126,110)(20,0)

\put(-120,35){\line(1,0){350}}

\put(-120,10){\line(1,0){350}}
\put(-120,0){\line(1,0){350}}

\put(-100,35){\thicklines\circle*{2.5}}
\put(-85,35){\thicklines\circle*{2.5}}
\put(-64,35){\thicklines\circle*{2.5}}
\put(-54,35){\thicklines\circle*{2.5}}
\put(-32,35){\thicklines\circle*{2.5}}
\put(-12,35){\thicklines\circle*{2.5}}
\put(12,35){\thicklines\circle*{2.5}}
\put(24,35){\thicklines\circle*{2.5}}
\put(38,35){\thicklines\circle*{2.5}}
\put(52,35){\thicklines\circle*{2.5}}
\put(72,35){\thicklines\circle*{2.5}}
\put(100,35){\thicklines\circle*{2.5}}
\put(117,35){\thicklines\circle*{2.5}}
\put(138,35){\thicklines\circle*{2.5}}
\put(156,35){\thicklines\circle*{2.5}}
\put(172,35){\thicklines\circle*{2.5}}
\put(198,35){\thicklines\circle*{2.5}}
\put(216,35){\thicklines\circle*{2.5}}
\put(244,33){$\xi\in\mathcal N(\Lambda)$}

\put(-100,10){\thicklines\circle*{2.5}}

\put(-85,10){\thicklines\circle*{2.5}}
\put(-64,10){\thicklines\circle*{2.5}}
\put(-54,10){\thicklines\circle*{2.5}}
\put(-32,10){\thicklines\circle*{2.5}}
\put(-12,10){\thicklines\circle*{2.5}}
\put(12,10){\thicklines\circle*{2.5}}
\put(24,10){\thicklines\circle*{2.5}}
\put(38,10){\thicklines\circle*{2.5}}
\put(52,10){\thicklines\circle*{2.5}}
\put(72,0){\thicklines\circle*{2.5}}
\put(100,10){\thicklines\circle*{2.5}}
\put(117,0){\thicklines\circle*{2.5}}
\put(138,10){\thicklines\circle*{2.5}}
\put(156,0){\thicklines\circle*{2.5}}
\put(172,0){\thicklines\circle*{2.5}}
\put(198,0){\thicklines\circle*{2.5}}
\put(216,0){\thicklines\circle*{2.5}}
\put(240,8){$\xi_1\in\mathcal N(\Lambda)$}
\put(240,-2){$\xi_0\in\mathcal N(\Lambda)$}

\put(-120,60){\line(1,0){350}}
\put(50,60){\line(0,1){50}}
\put(230,60){\thicklines\vector(1,0){.0001}}
\put(50,110){\thicklines\vector(0,1){.0001}}
\put(50,100){\thicklines\circle*{2.5}}
\put(230,64){$x$}
\put(43,97){$1$}
\put(54,102){$\theta(x)$}
\put(232,80){$\theta:\Lambda\to[0,1]$}
{\put(-116,100){\line(1,0){76}}
\qbezier(-40,100)(20,100)(60,85)
\qbezier(60,85)(130,60)(220,61)
}

\end{picture}
\caption{Illustration of the marked point process $\mathbb P^\theta$: at the top, we see the graph of a possible marking function $\theta$; in the middle, a possible configuration $\xi$ corresponding to the point process $\mathbb P$; at the bottom, possible associated mark $0$ and mark $1$ configurations $\xi_0$ and $\xi_1$ corresponding to $\mathbb P^\theta$.}
\label{figure: marked}
\end{center}
\end{figure}

	\medskip
	
	In the remaining part of this section, for the sake of simplicity, we will present our main results about these conditional ensembles only in the case where $\mathbb P$ is a DPP. We note however that most of our results hold  for more general point processes. The theorems stated below are thus special cases of more general results, stated in full generality and proved in later sections.

	\medskip 
	
	In the simplest case, we condition on the event that no points have mark $1$ (in other words, there are no observed particles). If this event has non-zero probability, then the resulting conditional point process, which we will denote as $\mathbb P^\theta_{|\emptyset}$, is defined in the classical sense, and configurations in this point process have support in $\Lambda\times\{0\}$. Hence, by omitting the marks, we can identify configurations in this point process with configurations on $\Lambda$, and identify $\mathbb P^\theta_{|\emptyset}$ with a point process on $\Lambda$. The following result about the point process transformation $\mathbb P\mapsto \mathbb P^\theta_{|\emptyset}$, which is part of the more general Theorem \ref{theorem:condempty} in Section \ref{section:2}, will be fundamental for our concerns.
	\begin{theorem}\label{theorem:intro1}
		Let $\mathbb P$ be the DPP with kernel $K$ of a locally trace class operator $\mathrm K$ and let $\theta:\Lambda\to[0,1]$ be measurable and such that \(\mathrm M_{\sqrt{\theta}+1_B}\mathrm K\mathrm M_{\sqrt{\theta}+1_{B}}\) is trace class for any bounded Borel set $B$, and \[\mathbb P^\theta(\xi_1(\Lambda)=0)=\det(1-\mathrm M_{\sqrt{\theta}}\mathrm K\mathrm M_{\sqrt{\theta}})>0.\] Then $\mathbb P^\theta_{|\emptyset}$ is also a DPP, defined by the kernel of the $L^2(\Lambda,\mu)$-operator 
		\begin{equation}\label{def:Ktheta}\mathrm M_{1-\theta}\mathrm K (1-\mathrm M_\theta \mathrm K)^{-1}.\end{equation}
	\end{theorem}
	
\begin{remark}
If the locally trace-class operator $K$ is self-adjoint, it induces a DPP if and only if $0\leq \mathrm K\leq 1$ \cite{Soshnikov}. In this case, the condition that \(\mathrm M_{\sqrt{\theta}+1_B}\mathrm K\mathrm M_{\sqrt{\theta}+1_{B}}\) is trace class for any bounded Borel set $B$, is equivalent to
\(\int (\sqrt{\theta(x)}+1_B(x))^2K(x,x)  \mathrm d\mu(x)<\infty,\)
which is automatically satisfied whenever 
\(\int \theta(x) K(x,x)  d\mu(x)<\infty,\)
if $K(x,x)$ is locally integrable. The trace class condition is then practical to verify in concrete situations.
However, for non self-adjoint operators $\mathrm K$, ${\Tr\,}\mathrm K<\infty$ does not imply $\mathrm K$ being trace class, and then the condition that
\(\mathrm M_{\sqrt{\theta}+1_B}\mathrm K\mathrm M_{\sqrt{\theta}+1_{B}}\)
is trace class cannot be verified directly by computing a trace. In such cases, one rather tries to prove that an operator is a composition of Hilbert-Schmidt operators, to prove that it is trace class.
\end{remark}
	
	\begin{remark}
	Since \[\mathrm K(1-\mathrm M_\theta\mathrm K)^{-1}=(1-\mathrm K\mathrm M_\theta)^{-1}\mathrm K=\mathrm K+\mathrm K\mathrm M_{\sqrt\theta}(1-\mathrm M_{\sqrt\theta}\mathrm K\mathrm M_{\sqrt\theta})^{-1}\mathrm M_{\sqrt\theta}\mathrm K,\] the operator $\mathrm K(1-\mathrm M_\theta\mathrm K)^{-1}$ indeed exists provided that $\det(1-\mathrm M_{\sqrt{\theta}}\mathrm K\mathrm M_{\sqrt{\theta}})>0$. If $\mathrm K$ is self-adjoint, the operator \eqref{def:Ktheta} is in general not self-adjoint, however  the operator
		\[\mathrm M_{\sqrt{1-\theta}}\mathrm K (1-\mathrm M_\theta \mathrm K)^{-1}\mathrm M_{\sqrt{1-\theta}}\]
		is self-adjoint, and it is readily verified that this operator induces the same DPP $\mathbb P^\theta_{|\emptyset}$.
		If $\mathrm K$ is a projection, then it is easily seen that \eqref{def:Ktheta} is equal to the conjugation $(1-\mathrm M_\theta \mathrm K)\mathrm K (1-\mathrm M_\theta \mathrm K)^{-1}$ of $\mathrm K$.
	\end{remark}

	The probability to observe a given non-empty finite configuration of points in the marked point process will typically be zero, but we can still, $\mathbb P^\theta$-a.s., condition on such events by making use of disintegration and reduced Palm measures (see Section \ref{section:2} for details). 
	Given a mark $1$ configuration $\mathbf v=\{v_1,\ldots, v_m\}$, we will denote this conditional ensemble, which we will define properly in Section \ref{sec:24} below, as $\mathbb P^\theta_{|\mathbf v}$.
	Before stating our main result about $\mathbb P^\theta_{|\mathbf v}$ in the case where $\mathbb P$ is a DPP, we need to introduce the reduced Palm measure $\mathbb P_v$ of $\mathbb P$ associated to a point $v\in\Lambda$. This represents the conditional ensemble obtained by first conditioning $\mathbb P$ on the event $ v\in\supp\,\xi$, and then removing the point $v$ from the configuration. If $\mathbb P$ is the DPP with kernel $K$ and if $K(v,v)>0$, then \cite{ShiraiTakahashi} $\mathbb P_v$ is also a DPP, with kernel 
	\begin{equation}K_v(x,y)=\frac{\det\begin{pmatrix}K(x,y)&K(x,v)\\K(v,y)&K(v,v)\end{pmatrix}}{K(v,v)}.
	\end{equation}
	Similarly, we can condition $\mathbb P$ on the presence of a finite number of distinct points $\mathbf v=\{v_1,\ldots, v_m\}$. This is consistent in the sense that the reduced Palm measure $\mathbb P_\mathbf v=\mathbb P_{v_1,\ldots, v_m}$ is, for \(\mu^{\otimes m}\)-a.e. \(\mathbf v\in\Lambda^m\) such that \(\det(K(v_\ell,v_k))_{\ell,k=1}^m>0\), equal to the measure $\left(\left(\mathbb P_{v_1}\right)_{v_2}\ldots\right)_{v_m}$ obtained by iteratively conditioning on $v_1,\ldots, v_m$, for any chosen order of the points.
{Let us for notational convenience write $K(\mathbf v,\mathbf v)$ for the $m\times m$ matrix $\begin{pmatrix}
		K(v_\ell,v_k)
	\end{pmatrix}_{\ell,k=1}^m$, $K(x,\mathbf v)$ for the row vector$\left(K(x,v_k)\right)_{k=1}^m$, and $K(\mathbf v,y)$	
for the column vector $\left(K(v_\ell,y)\right)_{\ell=1}^m$.
	If $\mathbb P$ is a DPP with kernel $K$ and if $\det K(\mathbf v,\mathbf v)>0$, then $\mathbb P_\mathbf v
	$ is the DPP with kernel given by  
	\begin{align}\label{def:Kv}
		K_\mathbf v(x,y)&=\frac{\det\begin{pmatrix}
				K(x,y) & K(x,\mathbf v) \\ K(\mathbf v,y) & K(\mathbf v,\mathbf v)
		\end{pmatrix}}{\det K(\mathbf v,\mathbf v)},
	\end{align}
	which defines a finite rank perturbation of \(\mathrm K\). Let us also set for consistency the convention that when \(\mathbf v=\emptyset\), \(\mathbb P_\emptyset=\mathbb P\) and \( K_\emptyset=K\). }
	
	\medskip
	
	In analogy to and as a generalization of Theorem \ref{theorem:intro1}, we have the following result, which is part of the more general Theorem \ref{theorem:condnonempty} below.
	\begin{theorem}\label{theorem:intro3}
		If $\mathbb P$ is the DPP with locally trace class operator $\mathrm K$ and \(\theta\in L^\infty(\Lambda,\mu)\) {is} such that \(\mathrm M_{\sqrt{\theta}+1_B}\mathrm K\mathrm M_{\sqrt{\theta}+1_{B}}\) is trace class for any bounded Borel set $B$, then for $\mathbb P^\theta$-a.e.\ $\xi_1$, writing $\mathbf v=\supp\,\xi_1$, we have 
		$\det\left(1-\mathrm M_{\sqrt\theta} \mathrm K_{\mathbf v}\mathrm M_{\sqrt\theta}\right)\neq 0$, and
		$\mathbb P^\theta_{|\bf v}$ is also a DPP, defined by the $L^2(\Lambda,\mu)$-operator 
		\begin{equation}\label{def:Kvtheta}
			\mathrm M_{1-\theta}\mathrm K_{\mathbf v}\left(1-\mathrm M_\theta \mathrm K_{\mathbf v}\right)^{-1}.
		\end{equation}
	\end{theorem}
	\begin{remark}
		This result implies that the class of DPPs is stable under the transformation $\mathbb P\mapsto \mathbb P^\theta_{|\mathbf v}$. More is actually true: as we will see, each of the subclasses of DPPs defined in Examples \ref{Example-OPE}--\ref{example:integrable} are also stable, and in Assumptions \ref{regularity assumptions} below, we will define a larger class of (not necessarily determinantal) point processes which is stable under this transformation.
	\end{remark}
	Section \ref{section:2} will be devoted to the rigorous construction of the marked and conditional point processes $\mathbb P^\theta, \mathbb P^\theta_{|\emptyset}, \mathbb P^\theta_{|\bf v}$, and to the proofs of (generalizations of) the results stated above.
	
	\medskip

	We should note that in the case where
	$\theta$ is the indicator function of a subset of $\Lambda$, all the above results are well-known, see e.g.\ \cite{BorodinSoshnikov, Bufetov, Bufetov3}.

	\medskip
	
\subsection{Rigidity}	In Section \ref{section:3}, we will study 
	conditional ensembles corresponding to infinite configurations of mark $1$ points $\delta_{\mathbf v}:=\sum_j\delta_{v_j}\in\mathcal N(\Lambda)$. In such cases, the disintegration theorem implies that one can still define $\mathbb P^\theta_{|\mathbf v}$.
	If $B\in \mathcal B_\Lambda$ is bounded and if $\theta=1_{B^c}$ is the indicator function of the complement of $B$, then $\mathbb P^\theta_{|\bf v}=\mathbb P^{1_{B^c}}_{|\bf v}$ is connected to the notion of number rigidity in the following manner. A point process $\mathbb P$ is said to be (number) rigid if for any bounded $B\in \mathcal B_\Lambda$, the conditional ensemble
	$\mathbb P^{1_{B^c}}_{|\bf v}$ has for $\mathbb P_1^{1_{B^c}}$-a.e.\ $\delta_{\mathbf v}$ a deterministic number of points, or in other words if there exists a  \(\mathcal C(\Lambda)\)-measurable function \[\ell:\mathcal N(\Lambda)\to\mathbb N\cup\{0,{\infty}\}:\delta_{\mathbf v} \mapsto\ell_{\mathbf v}\quad\mbox{ such that }\quad\mathbb P^{1_{B^c}}_{|\bf v}\left(\xi(B)=\ell_\mathbf v\right)=1.\]
	This property is trivially satisfied for DPPs defined by kernels of finite rank orthogonal projections, since the number of particles in these DPPs is deterministic.
	Remarkably, a wide class of DPPs defined by kernels of infinite rank locally trace class orthogonal projections are also known to be number rigid
 \cite{Ghosh, GhoshPeres, Bufetov4}. Conversely, it is known \cite{GhoshKrishnapur} that a DPP can only be number rigid if it is defined by a projection.
	The construction of marked and conditional ensembles naturally suggests the following stronger notion of rigidity, which requires that given a.e.\ configuration of mark $1$ points,  the number of mark $0$ point is deterministic.
	\begin{definition}\label{def:rigidity}
		A point process $\mathbb P$ is {\em marking rigid} if 
		for any Borel measurable $\theta:\Lambda\to[0,1]$,
		there exists a Borel measurable function \[\ell:\mathcal N(\Lambda)\to \mathbb N\cup\{0,\infty\}:\delta_{\mathbf v}=\sum_{i}\delta_{v_i}\mapsto \ell_{v_1,v_2,\ldots}=:\ell_{\mathbf v}\] such that the following holds:  for $\mathbb P_1^\theta$-a.e.\ $\delta_{\mathbf v}$, \[\mathbb P^{\theta}_{|\mathbf v}\left(\xi(\Lambda)=\ell_{\mathbf v}\right)=1.\]
	\end{definition}
Here $\xi(\Lambda)$ denotes the number of points of a random configuration $\xi$ in the set $\Lambda$.

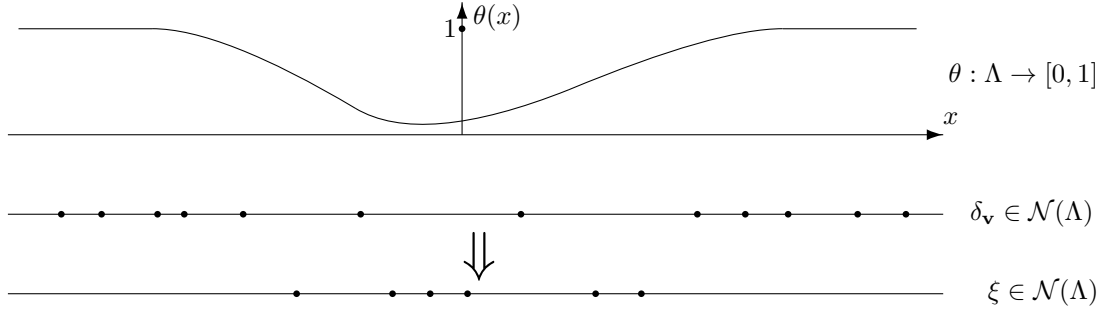
\begin{figure}[t]
\begin{center}
\begin{picture}(126,110)(20,0)
  
\put(-120,30){\line(1,0){350}}
\put(-120,0){\line(1,0){350}}

\put(-100,30){\thicklines\circle*{2.5}}

\put(-85,30){\thicklines\circle*{2.5}}
\put(-64,30){\thicklines\circle*{2.5}}
\put(-54,30){\thicklines\circle*{2.5}}
\put(-32,30){\thicklines\circle*{2.5}}
\put(-12,0){\thicklines\circle*{2.5}}
\put(12,30){\thicklines\circle*{2.5}}
\put(24,0){\thicklines\circle*{2.5}}
\put(38,0){\thicklines\circle*{2.5}}
\put(52,0){\thicklines\circle*{2.5}}
\put(72,30){\thicklines\circle*{2.5}}
\put(100,0){\thicklines\circle*{2.5}}
\put(117,0){\thicklines\circle*{2.5}}
\put(138,30){\thicklines\circle*{2.5}}
\put(156,30){\thicklines\circle*{2.5}}
\put(172,30){\thicklines\circle*{2.5}}
\put(198,30){\thicklines\circle*{2.5}}
\put(216,30){\thicklines\circle*{2.5}}
\put(240,28){$\delta_{\mathbf v}\in\mathcal N(\Lambda)$}
\put(240,-2){\ \ $\xi\in\mathcal N(\Lambda)$}

\put(50,9){\Huge{$\Downarrow$}}
\put(-120,60){\line(1,0){350}}
\put(50,60){\line(0,1){50}}
\put(230,60){\thicklines\vector(1,0){.0001}}
\put(50,110){\thicklines\vector(0,1){.0001}}
\put(50,100){\thicklines\circle*{2.5}}
\put(230,64){$x$}
\put(43,97){$1$}
\put(54,102){$\theta(x)$}
\put(232,80){$\theta:\Lambda\to[0,1]$}
{\put(-116,100){\line(1,0){50}}
\qbezier(-66,100)(-40,100)(10,70)
\qbezier(90,77)(145,100)(170,100)
\qbezier(10,70)(35,55)(90,77)

\put(170,100){\line(1,0){50}}
}

\end{picture}
\caption{Illustration of marking rigidity: at the top, we see the graph of a possible marking function $\theta$; at the bottom, a possible configuration of observed points $\delta_{\mathbf v}$ and a possible configuration $\xi$ in the conditional ensemble $\mathbb P^\theta_{|\mathbf v}$. If $\mathbb P$ is marking rigid, then the marking function $\theta$ and the observed configuration $\delta_\mathbf v$ a.s.\ determine the number of points ($6$ in the picture) in the unobserved configuration $\xi$.}
\label{figure: rigidity}
\end{center}
\end{figure}	
	
	The following result is a special case of Theorem \ref{theorem:rigidity}.
	\begin{theorem}\label{theorem:intro4}
		Let $\mathbb P$ 
		be a DPP induced by a locally trace class orthogonal projection $\mathrm K$ such that the following holds: for any $\epsilon>0$ and for any bounded $B\in\mathcal B_\Lambda$, there exists a bounded measurable function $f:\Lambda\to[0,+\infty)$ with bounded support such that \[\left.f\right|_{B}=1,\qquad {\rm Var }{\int_\Lambda f\dx\xi}<\epsilon,\]
		where ${\rm Var}$ denotes the variance with respect to $\mathbb P$.
		Then, $\mathbb P$ is marking rigid. 
	\end{theorem}
	\begin{remark}
		It is well-known that the existence of a function $f$ as in the above statement for any bounded $B\in\mathcal B_\Lambda$ and $\epsilon>0$, implies number rigidity of the point process $\mathbb P$ \cite{Ghosh, GhoshPeres}, and it is also known that such $f$ exists if $\mathbb P$ is a DPP with sufficiently regular $2$-integrable kernel defining an orthogonal projection, such as the sine, Airy, and Bessel point processes \cite{Bufetov4}. We thus prove that these point processes are marking rigid.
	\end{remark}
	\begin{remark}
		The above result is trivial for DPPs induced by finite rank orthogonal projections, which a.s.\ have a deterministic number of points. For DPPs associated to infinite rank orthogonal projections, which have a.s.\ configurations with an infinite number of points, it is striking that the observation of a random (possibly infinite) part of a configuration determines a.s.\ the number of unobserved points. 
	\end{remark}
	
	\medskip
	{
\subsection{Orthogonal polynomial ensembles}	
	
	In Section \ref{section:4}, we will focus on the OPEs from Example \ref{Example-OPE}, and we will show that conditional ensembles of OPEs are also OPEs, but with a deformed weight function, see Proposition \ref{prop:OPE}.
	As a consequence, for $\Lambda=\mathbb R$, we show that a large class of OPEs on the real line, which are eigenvalue distributions of unitarily invariant Hermitian random matrices, are in fact conditional ensembles of the GUE. We also give explicit expressions for the marginal distribution of the mark $0$ points, given the number of mark $1$ points. These are in general not DPPs, but do have a special structure involving Hankel determinants.
	
%	We will also show that interesting probabilistic quantities can be expressed explicitly as integrals of Toeplitz or Hankel determinants.
	
	\medskip
	
\subsection{DPPs with integrable kernels and Riemann-Hilbert problems}	
	
	In Section \ref{section:5}, we will consider DPPs associated to integrable kernels. We will show how we can characterize the kernels of the associated conditional ensembles in terms of Riemann-Hilbert problems via the IIKS method, and explain how this opens the door for asymptotic analysis and for deriving integrable differential equations associated to the conditional ensembles $\mathbb P^\theta_{|\mathbf v}$. We will also be able to interpret Jacobi's identity for Fredholm determinants in terms of the conditional ensembles $\mathbb P^\theta_{|\emptyset}$.

	}

	\medskip

	\section{Construction of marked and conditional processes}
	\label{section:2}

	\subsection{Preliminaries}
	We consider a measurable space \(\left(\Lambda,\mathcal B_\Lambda\right)\), where \(\Lambda\) is a complete separable metric space and \(\mathcal B_\Lambda\) its Borel \(\sigma\)-algebra. We denote by \(\mathcal N(\Lambda)\) the 
	set of locally finite Borel counting measures on $\Lambda$, and by $\mathcal C(\Lambda)$
	the \(\sigma\)-algebra generated by cylinder sets of the form
	\begin{equation*}
		C=\bigcap_{i=1}^n\{\xi\in \mathcal N (\Lambda)\ :\ \xi(B_i)=k_i\},%\quad\quad B_i\in\mathcal B(\Lambda),\ k_i\in\mathbb N,
	\end{equation*}
	where \(B_1,\ldots,B_n\in\mathcal B_\Lambda\) are disjoint and \(n,k_1,\ldots, k_n\) are non-negative integers. Note that we can identify $\mathcal N(\Lambda)$ with the space of locally finite sets of points, counted with multiplicity. For configurations of distinct points, this means that we identify the counting measure $\xi$ with its support.
	We consider a point process $\mathbb P$ on $\Lambda$, i.e.\ a probability measure on the complete separable metric space \((\mathcal N(\Lambda),\mathcal C(\Lambda))\).
	
	\medskip
	
	For disjoint sets $B_1,\ldots, B_n\in\mathcal B_\Lambda$ and non negative integers $k_1,\ldots, k_n$ such that $\sum_{j=1}^{n}k_j=m$, the $m$-th {\em factorial moment measure} $M_m$ of \(\mathbb P\) is the symmetric measure on \(\Lambda^m\) given by
	\begin{equation}\label{factorialmoment}
		M_m(B_1^{k_1}\times\cdots\times B_n^{k_n})=\mathbb E\xi(B_1)^{[k_1]}\ldots\xi(B_n)^{[k_n]},\quad \mbox{with}\quad l^{[k]}=l(l-1)\cdots (l-k+1),
	\end{equation}
	if the average exists.
	Similarly, the $m$-th {\em Janossy measure} of \(\mathbb P\) (encoding its finite dimensional distributions) associated to $B\in\mathcal B_\Lambda$ is the symmetric measure on \(B^m\) given by
	\begin{equation*}
		J_m^B(B_1^{k_1}\times\cdots\times B_n^{k_n})=k_1!\cdots k_n!\mathbb P\big(\xi(B)=m,\ \xi(B_j)=k_j\mbox{ for } j=1,...,n\big),
	\end{equation*}
	where \(\sum_{j=1}^nk_j=m\) and \(\sqcup_{j=1}^nB_j=B\). 
	
	\medskip
	
	Throughout this section, we will impose the following regularity assumptions on the point process $\mathbb P$ on $\Lambda$.
	\(\)
	\begin{assumptions}\label{regularity assumptions}\(\)\\
		There exists a locally finite positive Borel measure \(\mu\) on \(\Lambda\) such that:
		\begin{enumerate}
			\item the point process $\mathbb P$ is simple, i.e.\ for \(\mu\)-a.e.\ \(x\in\Lambda\), \(\mathbb P(\xi(\{x\})\leq 1)=1\);
			\item $\mathbb P$ admits correlation functions of all orders, i.e.\ for any positive integer \(m\) there exists a (symmetric) locally integrable function \(\rho_m:\Lambda^m\to[0,+\infty)\) with respect to the measure $\mu^{\otimes m}$ on $\Lambda^m$ such that
			\begin{equation*}
				\dx M_m=\rho_m\mathrm d^m\mu ;%\rho_m\dx \mu^{\otimes m}=:\rho_m d^m\mu;
			\end{equation*}
			\item for any bounded \(B\in \mathcal B_\Lambda\), there exists \(\epsilon_B>0\) such that
			\begin{equation*}
				\sum_{m=1}^\infty\frac{(1+\epsilon_B)^m}{m!}M_m(B^m)<\infty.
			\end{equation*}
		\end{enumerate}
	\end{assumptions}
	Under these assumptions, it is a classical fact  \cite{Lenard, Soshnikov} that the correlation functions $\rho_m$ uniquely determine the point process $\mathbb P$. 
	We also have \cite{Macchi} that for every bounded \(B\in\mathcal B_\Lambda\), there exist locally integrable Janossy densities \(j_m^B:\Lambda^m\to [0,+\infty)\) such that \(\dx J_m^B=j_m^B\dx^m \mu\). Note that \(j^B_m\) is only defined on \(B^m\), however under Assumptions \ref{regularity assumptions}, we have the identity 
	\begin{equation}\label{Macchi relation 1}
		j^B(\mathbf{x})=\sum_{n=0}^\infty\frac{(-1)^n}{n!}\int_{B^n}\rho(\mathbf x\sqcup \mathbf y)\dx^n\mu(\mathbf y),
	\end{equation}	
	which allows to extend $j^B_m$ to \(\Lambda^m\), since the series converges in the space of locally integrable functions on $\Lambda^m$. Here we abbreviated
	\[j^B(\mathbf{x}):=j_m^B(x_1,\ldots, x_m),\qquad \rho(\mathbf{x}):=\rho_m(x_1,\ldots, x_m),\]
	because we interpret $\mathbf{x}$ either as a vector with $m$ components $x_1,\ldots, x_m$ or as a configuration $\{x_1,\ldots, x_m\}$ of $m$ (not necessarily distinct) points; $\rho(\mathbf{x}\sqcup\mathbf{y})$ then means
	$\rho_{m+n}(x_1,...,x_m,y_1,...,y_n)\) with $\mathbf x=(x_1,\ldots, x_m)$, $\mathbf y=(y_1,\ldots, y_n)$.
	This notation in which we neglect the order of the variables is justified because $\rho_m$ and $j_m^B$ are symmetric in their variables.	
	Moreover, if Assumptions \ref{regularity assumptions} (3) holds also globally, i.e.\ for $B=\Lambda$, we have	
	\begin{equation}\label{Macchi relation 2}	\rho(\mathbf{x})=\sum_{n=0}^\infty\frac{1}{n!}\int_{\Lambda^n}j^\Lambda(\mathbf x\sqcup \mathbf y)\dx^n\mu(\mathbf y).
	\end{equation}
	The above formulas continue to hold for \(m=0\) by adopting the conventions 
	\[\Lambda^0=B^0=\{\emptyset\},\quad 
	J^B_0(\emptyset)=j^B(\emptyset)=\mathbb P(\xi(B)=0),\quad M_0(\emptyset)=\rho(\emptyset)=1,\quad \mu^{\otimes 0}=\delta_\emptyset.
	\]
	
	\medskip
	
	{Let us note first that the Poisson point process with locally bounded intensity $\rho:\Lambda\to[0,+\infty)$ on $(\Lambda,\mu)$ satisfies Assumptions \ref{regularity assumptions} if $\mu$ is non-atomic, with correlation functions given by \eqref{def:Poissoncorrelation}.}
	Our interest goes in particular to DPPs, characterized by the kernel \(K:\Lambda^2\to\mathbb C\) of a locally trace class operator \(\mathrm K\) on \(L^2(\Lambda,\mu)\). These point processes are simple \cite{Soshnikov}, and the correlation functions are locally integrable and given by 
	\[
	\rho(x_1,\ldots,x_m)=\det\begin{pmatrix}
		K(x_j,x_k)
	\end{pmatrix}_{j,k=1}^m.
	\]
	The average multiplicative functional is a Fredholm determinant, recall \eqref{eq:LaplaceFredholm0}--\eqref{eq:LaplaceFredholm}.
	In particular, by \eqref{def:Fredholm}, we have for any \(\epsilon>0\) and bounded \(B\in\mathcal B_\Lambda\) that
	\[\sum_{m=0}^\infty\frac{(1+\epsilon)^m}{m!}M_m(B^m)=\det(1+(1+\epsilon)\mathrm M_{1_B} \mathrm K \mathrm M_{1_B})<\infty.\]
	Hence, we can conclude that Assumptions \ref{regularity assumptions} are satisfied when $\mathbb P$ is a DPP induced by a locally trace class operator $\mathrm K$.
	
	\medskip

	\subsection{Bernoulli marking}
	
	Given a point process satisfying Assumptions \ref{regularity assumptions} and a measurable function $\theta:\Lambda\to[0,1]$, we now construct a marked point process $\mathbb P^\theta$ on $\Lambda\times\{0,1\}$, by assigning to each point $x\in \Lambda$ independently a random Bernoulli variable which takes the value $1$ with probability $\theta(x)$, and the value $0$ with probability $1-\theta(x)$. %Writing $\theta_1(x)=\theta(x)$, $\theta_0(x)=1-\theta(x)$, 
	Let us define the measures \(\nu_x^\theta\) and $\mu^\theta$ respectively on \(\{0,1\}\) and $\Lambda_{\{0,1\}}:=\Lambda\times\{0,1\}$ as
	\begin{equation}\label{mutheta}
		\nu_x^\theta=(1-\theta(x))\delta_0+\theta(x)\delta_1,\quad\quad\quad\dx\mu^{\theta}(x;b)=\dx\nu_x^\theta(b)\dx\mu(x),\qquad x\in\Lambda,\ b\in\{0,1\}.
	\end{equation}

	This marked point process $\mathbb P^\theta$ satisfies Assumptions \ref{regularity assumptions} with $\Lambda$ replaced by $\Lambda_{\{0,1\}}$ and $\mu$ by $\mu^\theta$.
	The correlation functions are then simply given by 
	\begin{equation}\label{def:correlationmarked}\rho_m^\theta((x_1,b_1),\ldots, (x_m,b_m))=\rho_m(x_1,\ldots, x_m),%\quad j_m^{B,\theta}((x_1,b_1),\ldots, (x_m,b_m))=j_m^B(x_1,\ldots, x_m),
	\end{equation}
	with respect to the measure $\mu^\theta$, and hence do not depend on the marks.
	As a direct consequence of the expression for the correlation functions, if the ground process $\mathbb P$ is determinantal and induced by the operator \(\mathrm K\) on \(L^2(\Lambda,\mu)\) with kernel $K:\Lambda^2\to\mathbb C$, then the marked point process \(\mathbb P^{\theta}\) is also determinantal, induced by the operator $\mathrm K^\theta$ on \(L^2(\Lambda_{\{0,1\}},\mu^{\theta})\)
	with kernel 
	\begin{equation}\label{def:Kmarked}K^\theta((x,b_x), (y,b_y)):=K(x,y),\end{equation}
	which is independent of the marks.
	
	\medskip
	
	Now for \(b\in \{0,1\}\) and for a marked configuration \(\xi_{0,1}\in \mathcal N(\Lambda_{\{0,1\}})\), we define \(\xi_b\in \mathcal N(\Lambda)\) by
	\begin{equation}\label{def:xib1}
		\xi_{b}(B)=\xi_{0,1}(B\times\{b\}),\quad B\in \mathcal B_\Lambda,
	\end{equation}
	i.e.\ $\xi_b$ is the configuration of points with mark $b$,
	or equivalently
	\begin{equation}\label{def:xib2}
		\xi_b=\sum_{j : b_j=b}\delta_{x_j},\quad\text{when}\quad\xi_{0,1}=\sum_{j}\delta_{(x_j,b_j)}=\xi_0\otimes\delta_0+\xi_1\otimes\delta_1.
	\end{equation}
	As explained in the introduction, we interpret \(\xi_{1}\) as the configuration of observed particles and \(\xi_{0}\) as the configuration of unobserved particles. If we define the Borel measures \(\mu^\theta_b\) on \(\Lambda\) for  \(b\in\{0,1\}\) by
	\begin{equation}\label{def:mubtheta}
		\dx\mu_b^{\theta}(x)=\theta_b(x)\dx\mu(x),\qquad \theta_1=\theta,\quad \theta_0=1-\theta,
	\end{equation}
then the point processes \(\mathbb P^{\theta}_b\), $b=0,1$, obtained from \(\mathbb P^{\theta}\) via transportation through the maps \(\xi_{0,1}\mapsto\xi_b\), or in other words the marginal distributions of the mark $b$ configurations, also satisfy Assumptions \ref{regularity assumptions} with correlations functions \(\rho_b^{\theta}(\mathbf x)=\rho(\mathbf x)\) with respect to \(\mu_b^{\theta}\) and Janossy densities for any bounded \(B\in\mathcal B_{\Lambda}\) given by
	\begin{equation}
		%\rho_0(\mathbf{x})=\sum_{n=0}^\infty\frac{1}{n!}\int_{\Lambda^n}j_0(\mathbf x\sqcup \mathbf y)\dx^n\mu_0(\mathbf y),\quad\quad 
		j_{b}^{\theta,B}(\mathbf{x})=\sum_{n=0}^\infty\frac{(-1)^n}{n!}\int_{B^n}\rho(\mathbf x\sqcup \mathbf y)\dx^n\mu_b^{\theta}(\mathbf y).\label{def:janossybmarked}
	\end{equation}
	Both point processes $\mathbb P_0^\theta$ and $\mathbb P_1^\theta$ on $\Lambda$ are random independent thinnings of the ground point process $\mathbb P$.
	If the ground process is determinantal and induced by the kernel of a locally trace class operator \(\mathrm K\) on \(L^2(\Lambda,\mu)\), then so is \(\mathbb P^{\theta}_b\) with the same kernel, but now with the corresponding operator acting on \(L^2(\Lambda,\mu_b^{\theta})\) \cite{LavancierMollerRubak}. 
	
	\medskip
	
	Summarizing the above, we have proved the following result.
	
	\begin{proposition}\label{prop:marked}
		Let $\mathbb P$ satisfy Assumptions \ref{regularity assumptions}, and let $\theta:\Lambda\to[0,1]$ be measurable.
		\begin{enumerate}\item The marked point process $\mathbb P^\theta$ satisfies Assumptions \ref{regularity assumptions} with $\Lambda$ replaced by $\Lambda_{\{0,1\}}$ and $\mu$ by $\mu^\theta$; for $b=0,1$, the component $\mathbb P_b^\theta$ satisfies Assumptions \ref{regularity assumptions} with $\mu$ replaced by $\mu_b^\theta$ ; in both cases the correlation functions are the same as those of the ground process \(\mathbb P\).
			\item If $\mathbb P$ is the DPP with kernel $K$ on $(\Lambda,\mu)$, then $\mathbb P^\theta$ is the DPP with kernel \(K^\theta\) on $(\Lambda_{\{0,1\}},\mu^\theta)$. For $b=0,1$, the component 
			$\mathbb P_b^\theta$ is the DPP with kernel $K$ on $(\Lambda,\mu_b^\theta)$. 
		\end{enumerate}
	\end{proposition}
	\begin{remark}
		Observe the analogy with the corresponding result if $\mathbb P$ is the Poisson point process with intensity $\rho:\Lambda\to[0,+\infty)$ with respect to $\mu$. Then $\mathbb P^\theta$ is the Poisson point process with intensity $\rho^\theta(x,b)=\rho(x)$ on $\Lambda_{\{0,1\}}$ with respect to $\mu^\theta$, and $\mathbb P^\theta_b$ is the Poisson point process on $\Lambda$ with intensity $\rho$ with respect to $\mu_b^\theta$.
	\end{remark}
	
	\subsection{Conditioning on an empty observation}
	Let us now assume, in addition to Assumptions \ref{regularity assumptions}, that the probability to have no mark $1$ particles is non-zero, i.e.\ 
	\begin{equation} \label{assumption:empty}\mathbb P^\theta(\xi_1(\Lambda)=0)=\mathbb E\prod_{x\in\supp\,\xi}(1-\theta(x))=L[\theta]>0,\end{equation}
	where we recall the definition of $L[.]$ from \eqref{eq:LaplaceFredholm}. 
	Then, we can condition $\mathbb P^\theta$ on the event $\xi_1(\Lambda)=0$
	in the classical sense and identify it with a point process on $\Lambda$ by identifying $\xi_{0,1}$ with $\xi_0$, to obtain the conditional point process \(\mathbb P^\theta_{|\emptyset}\) on $\Lambda$ defined by
	\begin{equation}\label{def:condprobempty}
		\mathbb P_{|\emptyset}^\theta\left(\xi \in C\right)=\frac{\mathbb P^\theta\left(\xi_{0,1}\mbox{ is such that }\xi_0\in C,\ \xi_1(\Lambda)=0\right)}{\mathbb P^\theta(\xi_1(\Lambda)=0)},\qquad C\in\mathcal C(\Lambda).\end{equation} 
	We write $L_{|\emptyset}^\theta$ for the average multiplicative functional \eqref{eq:LaplaceFredholm} corresponding to the probability $\mathbb P_{|\emptyset}^\theta$. For $\mathrm K$ locally trace class on $(\Lambda,\mu)$ such that $\mathrm M_{\sqrt\theta}\mathrm K\mathrm M_{\sqrt\theta}$ is trace class and  $\det(1-\mathrm M_{\sqrt\theta}\mathrm K \mathrm M_{\sqrt\theta})>0$, let us introduce $K_{|\emptyset}^\theta$ as the kernel of the integral operator 
			\[\mathrm K(1-\mathrm M_\theta \mathrm K)^{-1}:L^2(\Lambda,\mu)\to L^2(\Lambda,\mu),\qquad 
\mathrm K(1-\mathrm M_\theta \mathrm K)^{-1}f(x)=\int_\Lambda K_{|\emptyset}^\theta(x,y)f(y)\dx \mu(y),		
			\] 
	and $\mathrm K^\theta_{|\emptyset}$ as the operator with the kernel $K^\theta_{|\emptyset}$ on $L^2(\Lambda,\mu_0^\theta)$,
			\[\mathrm K^\theta_{|\emptyset}:L^2(\Lambda,\mu_0^\theta)\to L^2(\Lambda,\mu_0^\theta),\qquad 
\mathrm K^\theta_{|\emptyset}f(x)=\int_\Lambda K_{|\emptyset}^\theta(x,y)f(y)\dx\mu_0^\theta(y).
			\] 
	
	\begin{theorem}\label{theorem:condempty}
		Let $\theta:\Lambda\to[0,1]$ be measurable and let $\mathbb P$ be such that \(L[\theta]>0\).
		\begin{enumerate}\item  The point process \(\mathbb P^\theta_{|\emptyset}\) is well-defined and has average multiplicative functional
			\[L^\theta_{|\emptyset}[\phi]=\frac{L[1-(1-\phi)(1-\theta)]}{L[\theta]}.\]
			If in addition \(\mathbb P\) satisfies Assumptions \ref{regularity assumptions} and there exists \(\epsilon>1\) such that \(L[-\epsilon\theta]<\infty\), then so does \(\mathbb P^\theta_{|\emptyset}\), with correlations functions with respect to \(\mu^\theta_0\) given by
			\begin{equation*}
				\rho^\theta_{|\emptyset}(\mathbf x)=\frac{j_1^{\theta,\Lambda}(\mathbf x)}{\mathbb P^\theta(\xi_1(\Lambda)=0)}.
			\end{equation*}
			\item If 
			$\mathbb P$ is the DPP
			with kernel $K$ of a locally trace class operator $\mathrm K$ and \(\mathrm M_{\sqrt{\theta}+1_B}\mathrm K\mathrm M_{\sqrt{\theta}+1_{B}}\) is trace class for any bounded Borel set $B$, then 
			$\mathbb P^\theta_{|\emptyset}$ is the DPP on $(\Lambda,\mu)$ with kernel of the integral operator \eqref{def:Ktheta} acting on $L^2(\Lambda,\mu)$, or equivalently the DPP on $(\Lambda,\mu_0^\theta)$ with kernel $K_{|\emptyset}^\theta$.
Moreover, if \(\mathrm K\) is self-adjoint then \(\mathrm K_{|\emptyset}^\theta\) is self-adjoint.
		\end{enumerate}
	\end{theorem}
	\begin{proof}
		\begin{enumerate}\item
			By definition of conditional probability, for $\phi:\Lambda\to\mathbb R^+$ measurable, we have
			\begin{align*}L^\theta_{|\emptyset}[\phi]&=\mathbb E_{|\emptyset}^\theta \prod_{u\in\supp\,\xi_0}(1-\phi(u))=\frac{\mathbb E \prod_{x\in\supp\,\xi}(1-\phi(x))(1-\theta(x))}{\mathbb P^\theta(\xi_1(\Lambda)=0) }\\
				&=
				\frac{\mathbb E \prod_{x\in\supp\,\xi}(1-\phi(x))(1-\theta(x))}{\mathbb E \prod_{x\in\supp\,\xi}(1-\theta(x)) }=\frac{L[1-(1-\phi)(1-\theta)]}{L[\theta]}.
			\end{align*}
			Now if \(\mathbb P\) is simple, then so is \(\mathbb P^\theta\) and a fortiori so is \(\mathbb P^\theta_{|\emptyset}\), and the inequality for \(\phi\geq 0\)
			\begin{equation*}
				L^\theta_{|\emptyset}[-\phi]\leq\frac{L[-\phi]}{L[\theta]}
			\end{equation*}
			shows that \(\mathbb P^\theta_{|\emptyset}\) satisfies the third of Assumptions \ref{regularity assumptions} whenever \(\mathbb P\) does. It thus remains to compute the correlation functions. Note first that \(L[-\epsilon\theta]<\infty\) implies that \(\mathbb P^\theta_1\) satisfies the third of Assumptions \ref{regularity assumptions} with \(B=\Lambda\), so that the global Janossy densities \(j^{\theta,\Lambda}_1\) are well-defined and given by \eqref{Macchi relation 1}. The computations hereafter then involve absolutely convergent series, and all the needed results of integration theory may be applied. Let \(\eta=1-(1-\theta)(1-\phi)=\theta+(1-\theta)\phi\), then
			\begin{equation*}
				L[\eta]=\sum_{n\geq 0}\frac{(-1)^n}{n!}\int_{\Lambda^n}\rho_n(\mathbf x)\prod_{j=1}^n\eta(x_j)\dx^n\mu(\mathbf x).
			\end{equation*}
			Writing \(\mathbf x=\mathbf y\sqcup\mathbf z\) and using the symmetry of the measure \(\rho_n(\mathbf x)\dx^n\mu(\mathbf x)\) yields that each integral is equal to
			\begin{equation*}
				\begin{aligned}
					\sum_{l=0}^n\binom nl\int_{\Lambda^n}\rho_n(\mathbf y\sqcup\mathbf z)\prod_{j=1}^l(1-\theta(y_j))\phi(y_j)\prod_{i=1}^{n-l}\theta(z_i)\dx^n\mu(\mathbf y\sqcup\mathbf z),
				\end{aligned}
			\end{equation*}
			so that
			\begin{equation*}
				L[\eta]=\sum_{l\geq 0}\frac{(-1)^l}{l!}\int_{\Lambda^l}\left[\sum_{n\geq l}\frac{(-1)^{n-l}}{(n-l)!}\int_{\Lambda^{n-l}}\rho(\mathbf y\sqcup\mathbf z)\dx^{n-l}\mu^\theta_1(\mathbf z)\right]\prod_{j=1}^l(1-\theta(y_j))\phi(y_j)\dx^l\mu^\theta_0(\mathbf y).
			\end{equation*}
			We recognize expression \eqref{Macchi relation 1} for \(j_1^{\theta,\Lambda}\) in the integral. Dividing the previous equation by \(\mathbb P^\theta(\xi_1(\Lambda)=0)=L[\theta]\), we get an expression for \(L^\theta_{|\emptyset}[\phi]\), and when \(\phi=- 1_B\), this implies the existence of all factorial moment measure \(M^\theta_{m|\emptyset}\) of $\mathbb P^\theta_{|\emptyset}$, given the estimate \(M^\theta_{m|\emptyset}(B^{m})\leq L^\theta_{|\emptyset}[- 1_B]\). Replacing \(\phi\) by \(w\phi\) for \(w\in\mathbb C\) with a small enough modulus, we obtain a power series in \(w\) and we can read off the expressions for the correlation functions \(\dx M^\theta_{m|\emptyset}=\rho^\theta_{m|\emptyset}\dx^m\mu^\theta_0\) by looking at each power of \(w\).
			\item
			Let $(B_n)_{n\in\mathbb N}$ be an exhausting increasing sequence of bounded Borel subsets of $\Lambda$, and let $\mathrm K_n=\mathrm M_{1_{B_n}}\mathrm K\mathrm M_{1_{B_n}}$. 
By {\eqref{eq:LaplaceFredholm0}--\eqref{eq:LaplaceFredholm}}, the associated conditional ensemble $(\mathbb P_n)^\theta_{|\emptyset}$ has average multiplicative functional equal to 
			\begin{align*}\frac{\det\left(1-\mathrm M_{\phi+\theta-\phi\theta}\mathrm K_n\right)}{\det(1-\mathrm M_\theta \mathrm K_n)}&=\det\left[\left((1-\mathrm M_\theta \mathrm K_n)-\mathrm M_\phi \mathrm M_{1-\theta}\mathrm K_n\right)(1-\mathrm M_\theta \mathrm K_n)^{-1}\right]\\&=\det\left[1-\mathrm M_\phi \mathrm M_{1-\theta}\mathrm K_n(1-\mathrm M_\theta \mathrm K_n)^{-1}\right],\end{align*}
			and it follows that $(\mathbb P_n)^\theta_{|\emptyset}$ is also determinantal on $(\Lambda,\mu)$ with kernel of the integral operator
			$\mathrm M_{1-\theta}\mathrm K_n(1-\mathrm M_\theta \mathrm K_n)^{-1}$.
	The left hand side in the above identity is equal to $\frac{\det\left(1-\mathrm M_{\sqrt{\phi+\theta-\phi\theta}}\mathrm K_n\mathrm M_{\sqrt{\phi+\theta-\phi\theta}}\right)}{\det(1-\mathrm M_{\sqrt{\theta}} \mathrm K_n\mathrm M_{\sqrt{\theta}})}$ and as $n\to\infty$, it converges to 
\[\frac{\det\left(1-\mathrm M_{\sqrt{\phi+\theta-\phi\theta}}\mathrm K\mathrm M_{\sqrt{\phi+\theta-\phi\theta}}\right)}{\det(1-\mathrm M_{\sqrt{\theta}} \mathrm K\mathrm M_{\sqrt{\theta}})}=L^\theta_{|\emptyset}[\phi],\] since 	
$\mathrm M_{\sqrt{\phi+\theta-\phi\theta}}\mathrm K_n\mathrm M_{\sqrt{\phi+\theta-\phi\theta}}$ and $\mathrm M_{\sqrt{\theta}}\mathrm K_n\mathrm M_{\sqrt{\theta}}$	 converge in trace norm to 
$\mathrm M_{\sqrt{\phi+\theta-\phi\theta}}\mathrm K\mathrm M_{\sqrt{\phi+\theta-\phi\theta}}$ and $\mathrm M_{\sqrt{\theta}}\mathrm K\mathrm M_{\sqrt{\theta}}$, since
the latter two operators are trace class.
Indeed, 
$\mathrm M_{\sqrt{\theta}}\mathrm K\mathrm M_{\sqrt{\theta}}=\mathrm M_{\sqrt{\theta}+1_B}\mathrm K\mathrm M_{\sqrt{\theta}+1_B}$
with $B=\emptyset$; 
$\mathrm M_{\sqrt{\phi+\theta-\phi\theta}}\mathrm K\mathrm M_{\sqrt{\phi+\theta-\phi\theta}}$ can be decoomposed, with $B={\rm supp}\,\phi$, as
\[\mathrm M_{\sqrt{\phi+\theta-\phi\theta}}\mathrm 1_B\mathrm K\mathrm 1_B\mathrm M_{\sqrt{\phi+\theta-\phi\theta}}
+\mathrm 1_{B^c}\mathrm M_{\sqrt{\theta}}\mathrm K\mathrm M_{\sqrt{\theta}}\mathrm 1_{B^c}
\\+\mathrm 1_{B^c}\mathrm M_{\sqrt{\theta}}\mathrm K\mathrm 1_B\mathrm M_{\sqrt{\phi+\theta-\phi\theta}}
+\mathrm M_{\sqrt{\phi+\theta-\phi\theta}}\mathrm 1_B\mathrm K\mathrm M_{\sqrt{\theta}}\mathrm 1_{B^c}
\]
and it is easy to see that each term is trace class.

Similarly, the right hand side converges as $n\to\infty$ to 
			\[\det\left[1-\mathrm M_{\sqrt\phi} \mathrm M_{1-\theta}\mathrm K(1-\mathrm M_\theta \mathrm K)^{-1}\mathrm M_{\sqrt\phi}\right],\] since $\mathrm M_{\sqrt\phi} \mathrm M_{1-\theta}\mathrm K_n(1-\mathrm M_\theta \mathrm K_n)^{-1}\mathrm M_{\sqrt\phi}$ converges to $\mathrm M_{\sqrt\phi} \mathrm M_{1-\theta}\mathrm K(1-\mathrm M_\theta \mathrm K)^{-1}\mathrm M_{\sqrt\phi}$ in trace norm (note that we need the condition that \(\mathrm M_{\sqrt{\theta}+1_B}\mathrm K\mathrm M_{\sqrt{\theta}+1_{B}}\) is trace class for any bounded Borel set $B$ here again, in order to have $\mathrm M_{\sqrt{\phi}}\mathrm K\mathrm M_{\sqrt{\theta}}$, $\mathrm M_{\sqrt{\theta}}\mathrm K\mathrm M_{\sqrt{\phi}}$ trace class). Thus, $\mathbb P^\theta_{|\emptyset}$ is the DPP with kernel of the operator $\mathrm M_{1-\theta}\mathrm K(1-\mathrm M_\theta \mathrm K)^{-1}$ on $L^2(\Lambda,\mu)$, or equivalently the DPP on $(\Lambda,\mu_0^\theta)$ with kernel $K^\theta_{|\emptyset}$.\\
			If \(\mathrm K\) is self-adjoint on $L^2(\Lambda,\mu)$, then so is 
			\begin{equation*}
				\mathrm K(1-\mathrm M_\theta\mathrm K)^{-1}=\mathrm K+\mathrm K\mathrm M_{\sqrt\theta}(1-\mathrm M_{\sqrt\theta}\mathrm K\mathrm M_{\sqrt\theta})^{-1}\mathrm M_{\sqrt\theta}\mathrm K,
			\end{equation*}	
			as the sum of two self-adjoint operators, hence the kernel $K^\theta_{|\emptyset}$ defines a self-adjoint operator on $L^2(\Lambda,\mu_0^\theta)$ as well.			
		\end{enumerate}
	\end{proof}
	
	\begin{remark}
		If $\mathbb P$ is the Poisson point process with intensity $\rho$ on $\Lambda$ with respect to $\mu$, then $\mathbb P^\theta_{|\emptyset}$ is the Poisson point process with the same intensity $\rho$ on $\Lambda$, but with respect to $\mu_0^\theta$. Hence, $\mathbb P^\theta_{|\emptyset}$ is equal to $\mathbb P^\theta_0$, and as it should be, the fact that there are no mark $1$ points does not give any further information about the mark $0$ points. 
	\end{remark}
\begin{remark}
Theorem \ref{theorem:intro1} is a restatement of the second part of the above result.
\end{remark}	

	\medskip
	
	\subsection{Conditioning on a finite mark $1$ configuration $\xi_1$}\label{sec:24}
	
	For non-empty configurations $\xi_1$ of points with mark $1$, the situation is more involved.
	Here we need to assume that $\theta$ is such that
	there exists \(\epsilon>0\) such that \begin{equation}\label{condition:theta}L[-(1+\epsilon)\theta]={\mathbb E\prod_{x\in \supp\,\xi}(1+(1+\epsilon)\theta(x))<\infty},\end{equation} 
	where the average is with respect to the ground process $\mathbb P$.
	This condition ensures, by \eqref{factorialmoment}, that 
	$\mathbb P^\theta_1$ satisfies Assumptions \ref{regularity assumptions} (3) also for $B=\Lambda$, and in particular that 
	\[\mathbb E^\theta \xi_1(\Lambda)= \mathbb E\sum_{x\in\supp\,\xi}\theta(x)\leq {\mathbb E\prod_{x\in \supp\,\xi}(1+\theta(x))<\infty}.\]
	This implies that the number of observed particles
	\(\xi_1(\Lambda)\) is finite for \(\mathbb P^{\theta}\)-a.e.\ $\xi_1$. Based on such an observed configuration $\xi_1$, we would like to obtain information about the configuration $\xi_0$ of points with mark $0$. To this end, we want to define a point process \(\mathbb P^{\theta}_{|\mathbf v}\) on \(\Lambda\times\{0\}\) representing the restriction to $\Lambda\times\{0\}$ of the conditioning of \(\mathbb P^{\theta}\) on an observation \(\mathbf v=\{v_1,\ldots, v_m\}\), or more precisely on $\xi_1$ being equal to $\delta_\mathbf v:=\sum_{j=1}^m\delta_{v_j}$. We can then identify $\mathbb P^\theta_{|\mathbf v}$ with a point process on $\Lambda$ by omitting the marks $0$. The probability to observe given points $\mathbf v$ with mark $1$ will typically be zero, such that we cannot use classical conditional probability to construct the conditional point processes.
	
	\medskip
	
	\subsubsection*{Conditioning on $m$ mark $1$ points}
	Let us assume that $\mathbb P^{\theta}(\xi_1(\Lambda)=m)>0$. Then we can condition $\mathbb P^\theta$ on the event $\xi_1(\Lambda) =m$ in the classical sense.
	Now, we want to construct a family of conditional point processes $\left\{\mathbb P_{|\bf v}^\theta\right\}_{\bf v\in \Lambda^m}$, which is consistent in the sense that 
	averaging the $\mathbb P_{|\mathbf v}^\theta$-probability of an event \(\xi_{0}\in C\in\mathcal C(\Lambda)\) over the positions of the $m$-point configuration $v_1,\ldots, v_m$ (with respect to the probability $\mathbb P_1^\theta(.|\xi_1(\Lambda)=m)$) is equal to the $\mathbb P^\theta(.|\xi_1(\Lambda)=m)$-probability of the event $\xi_0\in C$. In other words, we average $v_1,\ldots, v_m$ with respect to the joint probability distribution
	\begin{equation}\label{def:jpdfm}
		\dx\pi^{\theta}_{1,m}(\mathbf v)=\dx \pi^{\theta}_{1,m}(v_1,\ldots, v_m):=\frac{j^{\theta,\Lambda}_{1,m}(v_1,\ldots, v_m)}{m!\mathbb P^{\theta}(\xi_1(\Lambda)=m)}\prod_{j=1}^m\dx\mu^{\theta}_1(v_j),
	\end{equation}
	for  $v_1,\ldots, v_m$, 
	where $j^{\theta,\Lambda}_{1,m}$ is the $m$-th order global Janossy density of the measure $\mathbb P^\theta_1$ for the mark $1$ configuration (which exists if \eqref{condition:theta} holds),
	and we will need consistency in the sense that
	\begin{equation}\label{proto disintegration}
		\int_{\Lambda^m}\mathbb P^{\theta}_{|\mathbf v}(\xi\in C)\dx\pi^{\theta}_{1,m}(\mathbf v)=\mathbb P^{\theta}\left(\left.\xi_{0}\in C \right|\ \xi_1(\Lambda)=m\right),\qquad C\in \mathcal C(\Lambda).
	\end{equation}
	
	\medskip
	
	\subsubsection*{Preliminaries on reduced Palm measures}
	
	As explained in Section \ref{section:1}, to construct $\mathbb P^\theta_{|\mathbf v}$, we need reduced local Palm distributions. 
	Given a point process $\mathbb P$ satisfying Assumptions \ref{regularity assumptions} and $m\in\mathbb N$, there exists a family of point processes
	$\left\{\mathbb P_{\mathbf w}\right\}_{\mathbf w\in\Lambda^m}$, which represent the conditioning of $\mathbb P$ on $m$ points $\mathbf w=\{w_1,\ldots, w_m\}\subset\supp\,\xi$, reduced by mapping $\xi\in\mathcal N(\Lambda)$ to its restriction $\left.\xi\right|_{\Lambda\setminus\mathbf w}$. We need the following fundamental properties (see e.g.\ \cite{DaleyVereJones}) of these $m$-th order reduced Palm measures. 
	\begin{enumerate}
		\item For any \(C\in\mathcal C(\Lambda) \), the map 
		\(\mathbf w\in\Lambda^m\mapsto\mathbb P_{\mathbf w}(C)\) is \(\mathcal B_{\Lambda^m}\)-measurable. \item For \(\mu^{\otimes m}\)-a.e.\ \(\mathbf w\in\Lambda^m\) such that \(\rho(\mathbf w)>0\), the reduced Palm measure 
		\(\mathbb P_{\mathbf w}\) satisfies Assumptions \ref{regularity assumptions}, and its correlation functions \(\rho_{\mathbf w}\) with respect to \(\mu\) are given by (see \cite{ShiraiTakahashi})
		\begin{equation}\label{Shirai-Takahashi theorem}
			\rho_{\mathbf w}(\mathbf x)=\frac{\rho(\mathbf x\sqcup\mathbf w)}{\rho(\mathbf w)}.
		\end{equation}
		\item Writing $\delta_{\mathbf w}=\sum_{j=1}^m\delta_{w_j}$, we have for any measurable \(\psi :\Lambda^m \times\mathcal N(\Lambda)\to\mathbb R^+\) that the disintegration 
		\begin{equation}\label{eq:Palmdisintegration}
			\mathbb E \sum\psi(\mathbf w;\xi-\delta_{\mathbf w})=\int_{\Lambda^m}\mathbb E_{\mathbf w}\psi(\mathbf w,\xi) \rho(\mathbf w) \dx^m\mu(\mathbf w)
		\end{equation}
		holds,
		where the sum at the left is over all ordered $m$-tuples $\mathbf w=(w_1,\ldots, w_m)$ of distinct points in $\supp\,\xi$ and
		where $\mathbb E_{\mathbf w}$ is the average with respect to $\mathbb P_{\mathbf w}$.
	\end{enumerate}

	\medskip
	In particular, the second property implies that if \(\mathbb P\) is determinantal with kernel $K$, then for $\mu^{\otimes m}$-a.e.\ \(\mathbf w\in\Lambda^m\) such that \(\det K(\mathbf w,\mathbf w)>0\), the reduced Palm measure $\mathbb P_{\mathbf w}$ is determinantal and induced by the kernel  $ K_{\bf w}$ given by \eqref{def:Kv}, or equivalently by
	\begin{align}\label{eq:Kv}
		K_{\mathbf w}(x,y)&=K(x,y)-K(x,\mathbf w)K(\mathbf w,\mathbf w)^{-1}K(\mathbf w, y),\end{align}
where we used the block determinant formula
\begin{equation}\label{eq:block}\det\begin{pmatrix}A&B\\C&D\end{pmatrix}=\det\left(A-BD^{-1}C\right)\det D,\end{equation}
and where similarly as before, $K(\mathbf w,\mathbf w)$ represents an $m\times m$ matrix, $K(x,\mathbf w)$ a row vector, and $K(\mathbf w,y)$ a column vector.

	\medskip
	
	\subsubsection*{Construction of the conditional ensembles}
	
	We will now apply the above properties of reduced Palm measures to the point process $\mathbb P^\theta\left(.|\xi_1(\Lambda)=m\right)$, the marked point process conditioned on observing exactly $m$ particles. {If $\mathbb P^\theta(\xi_1(\Lambda)=m)>0$, this point process indeed satisfies Assumptions \ref{regularity assumptions}. 
		{Setting $\mathbf w=\{(v_1,1),\ldots, (v_m,1)\}$ and $\mathbf v=\{v_1,\ldots, v_m\}$, we define $\mathbb P^\theta_{|\mathbf v}$ as the $m$-th order reduced local Palm distribution of $\mathbb P^\theta\left(.|\xi_1(\Lambda)=m\right)$ associated to the points $\mathbf w$. This is a point process on $\Lambda_{\{0,1\}}$ whose configurations have a.s.\ no points in $\Lambda\times\{1\}$; hence we can identify $\mathbb P^\theta_{|\mathbf v}$ with a point process on $\Lambda$ by omitting the marks.}
		%We will now prove some important properties of   the conditional ensembles $\mathbb P_{|\mathbf v}^\theta$,
{Before we prove some important properties of the conditional ensembles $\mathbb P_{.|\mathbf v}^\theta$, let us mention that another intuitive way of defining them would be to first take the Palm measure of \(\mathbb P^\theta\) at \(\mathbf w=((v_1,1),...,(v_m,1))\) and then condition on there being no other particles with mark $1$. The third item of the next result shows that this is indeed equivalent to our definition, and when $\mathbb P$ is a DPP, it has the advantage that it allows us to define the DPP $\mathbb P^\theta_{|\mathbf v}$ without need to pass via the (in general not determinantal) point process \(\mathbb P^\theta(.|\xi_1(\Lambda)=m)\).
		 Thus for $\mathrm K$ a locally trace class operator on $(\Lambda,\mu)$ such that $\mathrm M_\theta\mathrm K$ is trace class and  $\det(1-\mathrm M_\theta\mathrm K_\mathbf v)>0$, let us introduce $K_{|\mathbf v}^\theta$ as the kernel of the integral operator 
		\[\mathrm K_\mathbf v(1-\mathrm M_\theta \mathrm K_\mathbf v)^{-1}:L^2(\Lambda,\mu)\to L^2(\Lambda,\mu),\qquad 
		\mathrm K_\mathbf v(1-\mathrm M_\theta \mathrm K_\mathbf v)^{-1}f(x)=\int_\Lambda K_{|\mathbf v}^\theta(x,y)f(y)\dx \mu(y),		
		\] 
		and $\mathrm K^\theta_{|\mathbf v}$ as the operator with the kernel $K^\theta_{|\mathbf v}$ on $L^2(\Lambda,\mu_0^\theta)$,
		\[\mathrm K^\theta_{|\mathbf v}:L^2(\Lambda,\mu_0^\theta)\to L^2(\Lambda,\mu_0^\theta),\qquad 
		\mathrm K^\theta_{|\mathbf v}f(x)=\int_\Lambda K_{|\mathbf v}^\theta(x,y)f(y)\dx\mu_0^\theta(y).
		\] }
		\begin{theorem}\label{theorem:condnonempty}
			Let $\mathbb P$ satisfy Assumptions \ref{regularity assumptions}, and let $\theta:\Lambda\to[0,1]$ be measurable and such that \eqref{condition:theta} holds.	Let \(m\geq 0\) be such that \(\mathbb P^{\theta}(\xi_1(\Lambda)=m)>0\). The family of point processes \(\left\{\mathbb P^{\theta}_{|\mathbf{v}}\right\}_{\mathbf v\in\Lambda^m}\)
			satisfies the following properties.
			\begin{enumerate}
				\item
				For any \(C\in\mathcal C(\Lambda) \),  the map \(\mathbf v\in\Lambda^m\mapsto\mathbb P_{|\mathbf v}^\theta(C)\) is \(\mathcal B_{\Lambda^m}\)-measurable.
				\item {For any Borel measurable \(\phi :\mathcal N(\Lambda_{\{0,1\}})\to \mathbb [0,+\infty), \) writing $\delta_{\mathbf v}=\sum_{j=1}^m\delta_{v_j}$, we have the disintegration
					\begin{equation}\label{eq:disintegration}
						\mathbb E^{\theta}\left[\left.\phi(\xi_{0,1})\ \right|\ \xi_1(\Lambda)=m\right]=\int_{\Lambda^m}\mathbb E^{\theta}_{|\mathbf v}\phi(\xi\otimes\delta_{0}+\delta_{\mathbf v}\otimes\delta_{1})\dx\pi^{\theta}_{1,m}(\mathbf v),
					\end{equation}
					where $\pi^{\theta}_{1,m}$ is given by \eqref{def:jpdfm}.}
				\item For \(\pi_{1,m}^{\theta}\)-a.e.\ \(\mathbf v\in\Lambda^m\), the point process \(\mathbb P^{\theta}_{|\mathbf v}\) satisfies Assumptions \ref{regularity assumptions}, its correlation functions \(\rho^\theta_{|\mathbf v}\) with respect to \(\mu^\theta_0\) are given by 
					\begin{equation*}
						\rho^\theta_{|\mathbf v}(\mathbf x)=\frac{j^{\theta,\Lambda}_{1}(\mathbf x\sqcup\mathbf v)}{j^{\theta,\Lambda}_1(\mathbf v)},%v\rho^\theta_1(\mathbf x\mid\mathbf v):=
					\end{equation*}
					and its average multiplicative functional is given by
					\begin{align}
						&\label{eq:L}L^{\theta}_{|\mathbf v}[\phi_0]=\frac{L_{\mathbf v}[1-(1-\theta)(1-\phi_0)]}{L_{\mathbf v}[\theta]},
						%		\\
						%			&\rho_{|\mathbf v}^{\theta}(\mathbf u)=\rho_\mathbf v(\mathbf u)\frac{L_{\mathbf u\sqcup\mathbf v}[\theta]}{L_{\mathbf v}L[\theta]},
					\end{align}		
					where $L_{\mathbf v}$ denotes the average multiplicative functional of the reduced Palm measure $\mathbb P_{\mathbf v}$ of the ground process $\mathbb P$ on $\Lambda$ associated to the points $\mathbf v$.
					\item If \(\mathbb P\) is the DPP on $(\Lambda,\mu)$ with kernel of the operator \(\mathrm K\) on \(L^2(\Lambda,\mu)\), and \(\mathrm M_{\sqrt\theta +1_B}\mathrm K \mathrm M_{\sqrt\theta +1_B}\) is trace class for any bounded Borel set $B$,
					 then for \(\pi^{\theta}_{1,m}\)-a.e.\ \(\mathbf v\in\Lambda^m\), \(\mathbb P_{|\mathbf v}^{\theta}\) is the DPP on $(\Lambda,\mu)$ with kernel $(1-\theta)(x)K_{|\mathbf v}^\theta(x,y)$ of the operator
$\mathrm M_{1-\theta}\mathrm K_\mathbf v\left(1-\mathrm M_{\theta}\mathrm K_\mathbf v\right)^{-1}$ on $L^2(\Lambda,\mu)$, or equivalently the DPP on $(\Lambda,\mu_0^\theta)$ with kernel $K^\theta_{|\mathbf v}$.		
Moreover, if \(\mathrm K\) is self-adjoint on \(L^2(\Lambda,\mu)\) then the operator \(\mathrm K_{|\mathbf v}^\theta\) with kernel $K^\theta_{|\mathbf v}$ on \(L^2(\Lambda,\mu_0^\theta)\) is self-adjoint.
				\end{enumerate}
			\end{theorem}
			\begin{proof}\(\)\\
				\textit{(1)}\ This follows directly from the corresponding general property of reduced Palm measures. \\	\medskip
				\noindent\textit{(2)}\ {Applying  \eqref{eq:Palmdisintegration} to $\mathbb P=\mathbb P^\theta(.|\xi_1(\Lambda)=m)$
					and
					\[\psi:\Lambda_{\{0,1\}}^m\times \mathcal N(\Lambda_{\{0,1\}})\to[0,+\infty):(\mathbf w,\xi_{0,1})\mapsto\begin{cases}\phi(\xi_{0,1}+\delta_{\mathbf w})&\mbox{if $\xi_1(\Lambda)=0$ and $\mathbf w\in (\Lambda\times\{1\})^m$,}\\
						0&\mbox{otherwise,}
					\end{cases}
					\]
					with $\phi:\mathcal N(\Lambda_{\{0,1\}})\to [0,+\infty)$, and denoting $\mathbf w=((v_1,1),\ldots, (v_m,1))$, $\mathbf v=(v_1,\ldots, v_m)$,
					we obtain a family of point processes $\left\{\mathbb P_{|\mathbf v}^{\theta}\right\}_{\mathbf v\in\Lambda^m}$ on $\Lambda$
					which is such that
					\[
					m!\mathbb E^\theta \left[\left.\phi(\xi_{0,1})\ \right|\xi_1(\Lambda)=m\right]=
					\int_{\Lambda^m}\mathbb E^{\theta|m}_{\mathbf w}\phi(\xi\otimes\delta_{0}+\delta_{\mathbf v}\otimes\delta_{1}) \frac{j_{1,m}^{\theta,\Lambda}(\mathbf v)}{\mathbb P^\theta(\xi_1(\Lambda)=m)} \dx^m\mu_1^\theta(\mathbf v),
					\]
					where we used the symmetry of $\psi$ and the fact that there are $m!$ ordered $m$-tuples $\mathbf w$ in $\supp\,\xi_{0,1}$ at the left, and the fact that the $m$-point correlation function of $\mathbb P^\theta\left(.|\xi_1(\Lambda)=m\right)$ evaluated at $\mathbf w$ is equal to $\frac{j_{1,m}^{\theta,\Lambda}(\mathbf v)}{\mathbb P^\theta(\xi_1(\Lambda)=m)}$ (by \eqref{def:janossybmarked}) at the right.
					Using \eqref{def:jpdfm}, we obtain the required disintegration.}
				
				\medskip
				
				\noindent\textit{(3)}\ 
				{Let us apply \eqref{eq:disintegration} to the multiplicative statistic
					\begin{equation*}
						\phi(\xi_{0,1})=\prod_{(x,b)\in{\rm supp}\,\xi_{0,1}}(1-\phi_b(x)),
					\end{equation*}
					where
					$\phi_0,\phi_1:\Lambda\to (-\infty,1]$ are Borel measurable and \(\phi_0\) has bounded support. If \(\phi_1=0\), the disintegration implies that for \(\pi_{1,m}^{\theta}\)-a.e.\ \(\mathbf v\in\Lambda^m\), \(\mathbb P^{\theta}_{|\mathbf v}\) satisfies Assumptions \ref{regularity assumptions} \textit{(3)}, thereby justifying the computations hereafter involving series and integrals. }

				{The right hand side of \eqref{eq:disintegration} is then equal to 
					\begin{align*}
						&\int_{\Lambda^m}L^\theta_{|\mathbf v}[\phi_0] \prod_{j=1}^m(1-\phi_1(v))\dx \pi_{1,m}^\theta(\mathbf v)=\frac{1}{m!\mathbb P^\theta(\xi_1(\Lambda)=m)}\int_{\Lambda^m}L^\theta_{|\mathbf v}[\phi_0] \prod_{j=1}^m(1-\phi_1(v))j_{1}^{\theta,\Lambda}(\mathbf v)\dx^m\mu^{\theta}_1(\mathbf v)\\
						&=\frac{1}{m!\mathbb P^\theta(\xi_1(\Lambda)=m)}\int_{\Lambda^m}L^\theta_{|\mathbf v}[\phi_0] \prod_{j=1}^m(1-\phi_1(v))\left(\sum_{n=0}^\infty\frac{(-1)^n}{n!}\int_{\Lambda^n}\rho(\mathbf u\sqcup \mathbf v)\dx^n\mu_0^\theta(\mathbf u)\right)\dx^m\mu_{1}^\theta(\mathbf v)
						\\
						&=\frac{1}{m!\mathbb P^\theta(\xi_1(\Lambda)=m)}\int_{\Lambda^m}L^\theta_{|\mathbf v}[\phi_0] L_{\mathbf v}[\theta]\prod_{j=1}^m(1-\phi_1(v))\rho(\mathbf v)\dx^m\mu_{1}^\theta(\mathbf v),\end{align*}
					by \eqref{def:jpdfm}, \eqref{def:janossybmarked}, and \eqref{Shirai-Takahashi theorem}.}
				
				\medskip
				
				{The left hand side of \eqref{eq:disintegration} is equal to 
					\begin{align*}&\sum_{n=0}^\infty\frac{(-1)^n}{n!\ m!\mathbb P^\theta(\xi_1(\Lambda)=m)}\int_{\Lambda^m}\int_{\Lambda^n}\prod_{k=1}^n\phi_0(u_k) \rho(\mathbf u\sqcup \mathbf v)\dx^n\mu_0^\theta(\mathbf u)
						\prod_{j=1}^m(1-\phi_1(v_j))\dx^m\mu_1^\theta(\mathbf v)\\
						&=\frac{1}{m!\mathbb P^\theta(\xi_1(\Lambda)=m)}\int_{\Lambda^m}\left(\sum_{n=0}^\infty\frac{(-1)^n}{n!}\int_{\Lambda^n}\prod_{k=1}^n\phi_0(u_k) \rho_{\mathbf v}(\mathbf u)\dx^n\mu_0^\theta(\mathbf u)\right)
						\prod_{j=1}^m(1-\phi_1(v_j))\rho(\mathbf v)\dx^m\mu_1^\theta(\mathbf v)
						\\
						&=\frac{1}{m!\mathbb P^\theta(\xi_1(\Lambda)=m)}
						\int_{\Lambda^m}
						L_{\mathbf v}[1-(1-\theta)(1-\phi_0)]
						\prod_{j=1}^m(1-\phi_1(v_j))\rho(\mathbf v)\dx^m\mu_1^\theta(\mathbf v).
					\end{align*}
					Since both sides are equal for any choice of $\phi_1$, we can conclude that \eqref{eq:L} holds. To compute the correlation functions, we note that the transformation under consideration is the composition of taking the Palm measure and then conditioning on observing no particles, as the form of the average multiplicative functional reveals. Since for \(\pi^{\theta}_{1,m}\)-a.e \(\mathbf x\in\Lambda ^m\) one has \(\rho(\mathbf x)\geq j^{\theta,\Lambda}_1(\mathbf x)>0\) by \eqref{def:janossybmarked}, the result follows from the corresponding one in Theorem \ref{theorem:condempty} after noticing that the Janossy densities of the Palm measure are given by \(j^B_\mathbf v(\mathbf x)=\frac{j^B(\mathbf x\sqcup\mathbf v)}{\rho(\mathbf v)}\), while recalling the convention \(j^B(\emptyset)=\mathbb P(\xi(B)=0)\).  }
				\medskip
				\(\)\\ \noindent\textit{(4)}\ This follows after a straightforward computation from \textit{(3)} and \eqref{eq:LaplaceFredholm}. If \(\mathrm K\) is self-adjoint then so is \(\mathrm K_\mathbf v\), thus the result follows again from \textit{(3)} and Theorem \ref{theorem:condempty}.
			\end{proof}
			\begin{remark}The disintegration in part \textit{(2)} of the above result is more general than \eqref{proto disintegration}: it suffices indeed to take \[\phi(\xi_{0,1})=1_{C}(\xi_0)1_{\{\xi_1(\Lambda)=m\}}({\xi_{0,1}}),\]
				to recover \eqref{proto disintegration}.
			\end{remark}
			\begin{remark}
				For the Poisson point process on $\Lambda$ with intensity $\rho$, the above result is again trivial. We then have that $\mathbb P^\theta_{|\mathbf v}=\mathbb P^\theta_{|\emptyset}=\mathbb P^\theta_0$, in other words the positions of the mark $1$ points do not carry any information about the mark $0$ points.
				\begin{remark}
The last part of the above result implies Theorem \ref{theorem:intro3}.
\end{remark}	

			\end{remark}

	}		
			
			\medskip
			\ \medskip
			
			\section{Number rigidity and DPPs corresponding to projection operators}\label{section:3}
			
			\subsection{DPPs induced by orthogonal projections}
			Let $\mathbb P$ be a DPP on $\Lambda$, defined by a correlation kernel $K$ with respect to a locally finite positive Borel measure $\mu$ which is such that the associated operator $\mathrm K:L^2(\Lambda,\mu)\to L^2(\Lambda,\mu)$ is a locally trace class orthogonal projection, i.e.\ $0\leq\mathrm K\leq 1$ and $\mathrm K^2=\mathrm K$, onto a closed subspace $H$ of $L^2(\Lambda,\mu)$.
			The rank of $\mathrm K$ can be finite or infinite, but the results in this section will only be non-trivial in the infinite rank case.
			We assume here that the kernel $K:\Lambda^2\to \mathbb C$ of $\mathrm K$ is such that $\mathrm Kf(x)$ is defined for every $x\in\Lambda$ and for every $f\in L^2(\Lambda,\mu)$. Note that this is true whenever $K(x,.)\in L^2(\Lambda,\mu)$ for every $x\in\Lambda$, by the Cauchy-Schwarz inequality.
			Classical examples of admissible point processes are the sine, Airy, and Bessel point processes on the real line.
			
			\medskip
			
			By Proposition \ref{prop:marked}, the marked point process associated to $\mathbb P$ with marking function $\theta$ is the DPP on $(\Lambda_{\{0,1\}},\mu^\theta)$ with correlation kernel
			$K^\theta((x,b),(x',b'))=K(x,x')$, where we recall that $\mu^\theta$ is given by \eqref{mutheta}.
			The induced operator $\mathrm K^\theta$ acting on $L^2(\Lambda_{\{0,1\}},\mu^\theta)$ is the orthogonal 
			projection operator onto the space \[H^\theta:=\left\{h_{0,1}\in L^2(\Lambda_{\{0,1\}},{\mu^\theta}):h_{0,1}(.,0)=h_{0,1}(.,1)\in H\right\},\] and it is straightforward to verify that ${\rm dim}\, H^\theta={\rm dim}\, H$.
			
			\medskip

			As in Section \ref{section:2}, let us consider the conditional measure obtained by conditioning the marked point process on a configuration of mark $1$ points. 
			Under the assumptions that $\mathbb P^\theta(\xi_1(\Lambda)=m)>0$ and \(\mathrm M_{\sqrt\theta +1_B}\mathrm K \mathrm M_{\sqrt\theta +1_B}\) is trace class for any bounded Borel set $B$, we know from Theorem \ref{theorem:condnonempty} (4) that for $\pi_{1,m}^\theta$-a.e.\ $\mathbf v\in \Lambda^m$, the conditional measure $\mathbb P^\theta_{|\mathbf v}$ is the DPP induced by the operator
			\[\mathrm M_{1-\theta}\mathrm K_{\mathbf v}(1-\mathrm M_\theta \mathrm K_{\mathbf v})^{-1}
			=(1-\mathrm M_{\theta})\mathrm K_{\mathbf v}(1-\mathrm M_\theta \mathrm K_{\mathbf v})^{-1}
			\]
			on $L^2(\Lambda,\mu)$.
			Moreover, from \eqref{def:Kv}, it is straightforward to verify that $\mathrm K_{\mathbf v}$ is the orthogonal projection on 
			the subspace 
			\begin{equation}\label{def:Hv}H_{\mathbf v}=\overline{\{h\in H: h(v)=0\ \forall v\in \mathbf v\}}.\end{equation}
			Consequently, since $\mathrm K_{\mathbf v}^2=\mathrm K_{\mathbf v}$, the $L^2(\Lambda,\mu)$-operator $\mathrm M_{1-\theta}\mathrm K_{\mathbf v}(1-\mathrm M_\theta \mathrm K_{\mathbf v})^{-1}$ inducing $\mathbb P^\theta_{|\mathbf v}$
			is equal to a conjugation of $\mathrm K_{\mathbf v}$,
			\[(1-\mathrm M_{\theta}\mathrm K_{\mathbf v})\mathrm K_{\mathbf v}(1-\mathrm M_\theta \mathrm K_{\mathbf v})^{-1},
			\]
			and this implies that it is a (not necessarily self-adjoint) projection onto the subspace
			\begin{equation}\label{def:Hvtheta}H_{\mathbf v}^\theta:=\overline{(1-\mathrm M_\theta \mathrm K_{\mathbf v})H_{\mathbf v}}=\overline{(1-\mathrm M_\theta)H_{\mathbf v}},\end{equation}
			with dimension equal to that of $H_{\mathbf v}$,
			and that the $L^2(\Lambda,\mu^\theta_0)$-operator $\mathrm K^\theta_{|\mathbf v}$ is the orthogonal projection onto $H_{\mathbf v}$. Indeed, $\mathrm K^\theta_{|\mathbf v}$ is Hermitian, 
and
for $h\in H_{\mathbf v}$, we have
			\begin{equation}\label{eq:proj}\mathrm K^\theta_{|\mathbf v} h=\mathrm K_{\mathbf v}(1-\mathrm M_\theta\mathrm K_{\mathbf v})^{-1}\mathrm M_{1-\theta}h=
\mathrm K_{\mathbf v}(1-\mathrm M_\theta\mathrm K_{\mathbf v})^{-1}\mathrm M_{1-\theta}\mathrm K_{\mathbf v}h			
			=\mathrm K_{\mathbf v}h=h.\end{equation}

\medskip

Let us now consider the more general case where $\mathbb P$ is induced by a not necessarily Hermitian projection operator, say $\mathrm K=\mathrm P_{H,J}$ is the unique linear
projection  with range $H$ and kernel $J^{\perp}$, where 
\(H,J\) are closed subspaces of \(L^2(\Lambda,\mu)\) such that \(H\oplus J^{\perp}=L^2(\Lambda,\mu)\).	Note that the adjoint projection is given by \(\mathrm P_{H,J}^*=\mathrm P_{J,H}\). Since \(\phi\in L^2(\Lambda,\mu)\) can be identified with \(\phi\in L^2(\Lambda,\mu^\theta_0)\), we can also see \(H,J\) as subspaces of \(L^2(\Lambda,\mu^\theta_0)\). Examples of DPPs induced by non-Hermitian projections are biorthogonal ensembles and their scaling limits like the Pearcey DPP.
			\begin{proposition}\label{prop:proj}
				If \(\mathrm K=\mathrm P_{H,J}\), then 
				\begin{equation*}
					\mathrm K^\theta_{|\mathbf v}=\mathrm P_{H_\mathbf v,J_\mathbf v},
				\end{equation*}
				where \( H_\mathbf v, J_\mathbf v\) are seen as closed subspaces of \(L^2(\Lambda,\mu^\theta_0)\).
			\end{proposition}
			\begin{proof}
				First we recall that the transformation \(\mathrm K^\theta_{|\mathbf v}\) is obtained by first taking the reduced Palm measure and then conditioning on $\xi_1=\emptyset$ in view of Theorem \ref{theorem:condnonempty}, so that it suffices to prove the result separately in the cases \(\theta=0\) and \(\mathbf v=\emptyset\). The case $\theta=0$ is straightforward from \eqref{def:Kv}, while for $\mathbf v=\emptyset$, $\mathrm K^\theta_{|\emptyset}$ is a projection with range $H$ by \eqref{eq:proj} (with $\mathbf v=\emptyset$; observe that these equalities continue to hold when $\mathrm K$ is not self-adjoint).
				Finally, to identity the kernel, it suffices to apply the previous reasoning to \((\mathrm K^\theta_{|\emptyset})^*=(\mathrm K^*)^\theta_{|\emptyset}\).
			\end{proof}
			\(\)\\
			DPPs induced by projections have the property that the number of points in a configuration is almost surely equal to the rank of the projection \cite{Soshnikov}. If $\mathbb P$ has configurations with a deterministic number of points, it is obvious that the same must hold for $\mathbb P^\theta_{|\mathbf v}$, for any finite configuration $\mathbf v$. 
			Since the projection $P_{H_{\mathbf v}, J_{\mathbf v}}$ is also defined for infinite configurations $\mathbf v$, it is natural to ask whether the DPP induced by this projection can in such a situation still be interpreted as the conditional DPP $\mathbb P_{|\mathbf v}^\theta$. This is not true in general, see e.g.\  \cite{Bufetov5}, but we will see below that $\mathbb P^\theta_{|\mathbf v}$ is under suitable assumptions induced by an orthogonal projection, albeit not necessarily equal to $P_{H_{\mathbf v}, J_{\mathbf v}}$.

			\subsection{Disintegration}
			
			We first show that the family of conditional ensembles $\left\{\mathbb P^\theta_{|{\mathbf v}}\right\}_{\delta_{\mathbf v}\in\mathcal N(\Lambda)}$ exists under general conditions, and then we rely on results from \cite{BufetovQiuShamov} to prove that $\mathbb P^\theta_{|\mathbf v}$ is a DPP induced by a Hermitian operator $\mathrm K^\theta_{|\mathbf v}$ if $\mathbb P$ is a DPP induced by an orthogonal projection.
			\begin{proposition}\label{prop:conditionalinf}
				Let $\theta:\Lambda\to [0,1]$ be measurable, and let $\mathbb P$ satisfy Assumptions \ref{regularity assumptions}. There exists a family of point processes $\left\{\mathbb P_{|{\mathbf v}}^\theta\right\}_{\delta_{\mathbf v}\in\mathcal N(\Lambda)}$ such that the following conditions hold.
				\begin{enumerate}
					\item The map 
					\(\delta_{\mathbf v}\in\mathcal N(\Lambda)\mapsto\mathbb P_{|{\mathbf v}}^\theta(C)\) is \(\mathcal C(\Lambda)\)-measurable for any \(C\in\mathcal C(\Lambda) \), and the
					disintegration
					\[\mathbb P_0^\theta(C)=\int_{\mathcal N(\Lambda)}\mathbb P_{|{\mathbf v}}^\theta(C) \dx\mathbb P_1^\theta(\delta_{\mathbf v})\]
					holds for any $C\in\mathcal C(\Lambda)$.
					\item If $\mathbb P$ is a DPP induced by an orthogonal projection with kernel
					$K:\Lambda^2\to \mathbb R$ such that $\mathrm Kf(x)$ is defined for every $x\in\Lambda$ and for every $f\in L^2(\Lambda,\mu)$, then for $\mathbb P_1^\theta$-a.e.\  $\delta_{\mathbf v}$, $\mathbb P_{|{\mathbf v}}^\theta$ is a DPP induced by a Hermitian locally trace class operator $\mathrm K_{|{\mathbf v}}^\theta$.
				\end{enumerate} 
			\end{proposition}
			\begin{proof}
				Let us define $\widehat{\mathbb P}_{|{\mathbf v}}^\theta$ by disintegrating $\mathbb P^\theta$ with respect to the surjective mapping
				\[r:\mathcal N(\Lambda_{\{0,1\}})\to\mathcal N(\Lambda\times\{1\}):\xi_0\otimes \delta_0+\xi_1\otimes \delta_1\mapsto \xi_1\otimes\delta_1.\]
				The disintegration theorem then implies that the map $\delta_{\mathbf v}\mapsto \widehat{\mathbb P}_{|{\mathbf v}}^\theta(\widetilde C)$ is Borel measurable for any $\widetilde C\in\mathcal C(\Lambda_{\{0,1\}})$, and that
				\[\mathbb P^\theta(\widetilde C)=\int_{\mathcal N(\Lambda_{\{0,1\}})} \widehat{\mathbb P}_{|{\mathbf v}}^\theta(\widetilde C)\dx\mathbb P^\theta(r^{-1}(\delta_{\mathbf v}\otimes\delta_{1})).\]
				Taking $\widetilde C=C\otimes\{\delta_0\}\subset\mathcal N(\Lambda\times\{0\})$ with $C\in \mathcal C(\Lambda)$ and defining $\mathbb P^\theta_{|{\mathbf v}}$ on $\Lambda$ as
				${\mathbb P}_{|\mathbf v}^\theta(C):=\widehat{\mathbb P}_{|{\mathbf v}}^\theta(\widetilde C),
				$
				this becomes
				\[\mathbb P^\theta_0(C)=\int_{\mathcal N(\Lambda)} {\mathbb P}_{|{\mathbf v}}^\theta(C)\dx\mathbb P^\theta_1(\delta_{\mathbf v}),\]
				and part {\em (1)} of the theorem is proved.

				\medskip
				
				Part (2) follows directly upon applying  \cite[Lemma 1.11]{BufetovQiuShamov} to the marked point process $\mathbb P^\theta$ and $W=\Lambda\times\{1\}$. 
			\end{proof}
			\begin{remark}
				It is important to note that the operator $\mathrm K^\theta_{|\mathbf v}$ is not necessarily a projection in part (2) of the above result.
			\end{remark}
			
			\subsection{Marking rigidity}
			
			We will now further refine our assumptions on $\mathbb P$, in order to obtain a sufficient condition for $\mathbb P$ to be marking rigid. Let us emphasize that we will not need $\mathbb P$ to be a DPP. However, DPPs induced by integrable orthogonal projection operators will provide our main example of point processes which satisfy the assumption below.
			
			\begin{assumptions}\label{assumption rigid}
				$\mathbb P$ satisfies Assumptions \ref{regularity assumptions} and is such that the following holds: for any $\epsilon>0$ and for any bounded $B\in\mathcal B_\Lambda$, there exists a bounded measurable function $f:\Lambda\to[0,+\infty)$ with bounded support such that \[\left.f\right|_{B}=1,\qquad {\rm Var}{\int_\Lambda f\dx\xi<\epsilon},\]
				where ${\rm Var}$ denotes the variance with respect to $\mathbb P$.
			\end{assumptions}

			\begin{theorem}\label{theorem:rigidity}
				Let $\mathbb P$ satisfy Assumptions \ref{assumption rigid}. \begin{enumerate}
					\item If for any measurable $\theta:\Lambda\to[0,1]$, $\mathbb P^\theta(\xi_0(\Lambda)<\infty)$ is either $0$ or $1$, then $\mathbb P$ is marking rigid.
					\item Let $\mathbb P$ be a DPP induced by a locally trace class orthogonal projection
					with kernel
					$K:\Lambda^2\to \mathbb C$ such that $\mathrm Kf(x)$ is defined for every $x\in\Lambda$ and for every $f\in L^2(\Lambda,\mu)$. Then $\mathbb P$ is marking rigid, and for any measurable $\theta:\Lambda\to[0,1]$ such that $\mathrm M_{1-\theta}\mathrm K$ is trace class, the conditional ensemble $\mathbb P^\theta_{|\mathbf v}$ is for $\mathbb P^\theta_1$-a.e.\ $\delta_{\mathbf v}$ induced by a finite rank orthogonal projection $\mathrm K^\theta_{|\mathbf v}$.\end{enumerate}
			\end{theorem}
			\begin{proof}
{	
			Let us first consider the case where $\theta$ is such that $\mathbb P^\theta(\xi_0(\Lambda)<\infty)=0$, when $\mathbb P^\theta$-a.s., we have $\xi_0(\Lambda)=\infty$. Then, by Proposition \ref{prop:conditionalinf}, 
				\[1=\mathbb P_0^\theta(\xi(\Lambda)=\infty)=\int_{\mathcal N(\Lambda)}\mathbb P^\theta_{|\mathbf v}(\xi(\Lambda)=\infty)\dx\mathbb P_1^\theta(\delta_{\mathbf v}),\]
				and consequently $\mathbb P^\theta_{|{\mathbf v}}(\xi(\Lambda)=\infty)=1$ for $\mathbb P_1^\theta$-a.e.\ $\delta_{\mathbf v}$.
				
				\medskip
				
				We now assume that $\mathbb P^\theta(\xi_0(\Lambda)<\infty)=1$. Let $\Lambda_1\subset\Lambda_2\subset\cdots$ be an exhausting sequence of bounded Borel subsets of $\Lambda$.
				First, we observe that for $\mathbb P^\theta$-a.s. $\xi_0$, $\xi_0(\Lambda\setminus \Lambda_n)= 0$ for $n$ sufficiently large.
				Secondly, we take a sequence of positive numbers $\epsilon_1,\epsilon_2,\ldots$ which converges to $0$ as $n\to\infty$, and we observe that by Assumptions \ref{assumption rigid}, there exists a sequence of bounded measurable functions $f_1,f_2,\ldots$ with bounded support such that $\left.f_n\right|_{\Lambda_n}=1$, and such that
				${\rm Var }\int_\Lambda f_n\dx\xi<\epsilon_n$.
				This implies 
				by Chebyshev's inequality that
				\[\mathbb P\left(\left|\int_\Lambda f_n\dx\xi-\mathbb E \int_\Lambda f_n\dx\xi\right|\geq\delta\right)\leq \delta^{-2}{\rm Var}\,\int_\Lambda f_n\dx\xi\leq \delta^{-2}\epsilon_n\to 0,\]
				as $n\to\infty$ for any $\delta>0$, and hence 
				that 
there exists a subsequence $\left(f_{n_j}\right)_{j\in\mathbb N}$				
				such that
				\[\lim_{j\to\infty}\left(\int_\Lambda f_{n_j}\dx\xi-\mathbb E \int_\Lambda f_{n_j}\dx\xi\right)=0,\qquad \mbox{for $\mathbb P$-a.e.\ $\xi$},\] 
or equivalently with \textcolor{orange}{}\(\hat f_n(x,b)=f_n(x)\)
\[
\lim_{j\to\infty}\left(\int_{\Lambda_{\{0,1\}}} \hat f_{n_j}\dx\xi_{0,1}-\mathbb E^\theta \int_{\Lambda_{\{0,1\}}} \hat f_{n_j}\dx\xi_{0,1}\right)=0,\qquad \mbox{for $\mathbb P^\theta$-a.e.\ $\xi_{0,1}$}.
%\lim_{j\to\infty}\left(\int f_{n_j}(x)\dx(\xi_0+\xi_1)(x)-\mathbb E \int f_{n_j}(x)\dx\xi(x)\right)=0,\qquad \mbox{for $\mathbb P^\theta$-a.e.\ $\xi_{0,1}$}.
\] 
				
				\medskip
				
For any $\xi_{0,1}\in\mathcal N(\Lambda_{\{0,1\}})$, we can write $\xi_0(\Lambda)$ as
				\begin{align*}\xi_0(\Lambda)&
%					=\int f_{n_j}(x)\dx(\xi_{0}+\xi_1)(x)+\int (1-f_{n_j}(x))\dx\xi_0(x)-\int f_{n_j}(x)\dx\xi_1(x)\\
%					&
%					=\left(\int f_{n_j}(x)\dx(\xi_0+\xi_1)(x)-\mathbb E\int f_{n_j}(x)\dx\xi(x)\right)+\mathbb E\int f_{n_j}(x)\dx\xi(x)+\int (1-f_{n_j}(x))\dx \xi_0(x)-\int f_{n_j}(x)\dx \xi_1(x).
					=\int_{\Lambda_{\{0,1\}}}1_{\{b=0\}}\dx\xi_{0,1}(x,b)=\int_{\Lambda_{\{0,1\}}}\hat f_{n_j}\dx\xi_{0,1}+\int_\Lambda (1-f_{n_j})\dx\xi_0-\int_\Lambda f_{n_j}\dx\xi_1\\
					&
					=\left(\int_{\Lambda_{\{0,1\}}} \hat f_{n_j}\dx\xi_{0,1}-\mathbb E^\theta \int_{\Lambda_{\{0,1\}}} \hat f_{n_j}\dx\xi_{0,1}\right)+\mathbb E\int_\Lambda f_{n_j}\dx\xi+\int_\Lambda (1-f_{n_j})\dx \xi_0-\int_\Lambda f_{n_j}\dx \xi_1.
				\end{align*}
Taking the limit $j\to\infty$,				 the part between parentheses at the right converges to $0$ $\mathbb P^\theta$-a.s., and
				the term $\int_\Lambda (1-f_{n_j})\dx \xi_0$ vanishes $\mathbb P^\theta$-a.s.\, for sufficiently large $j$, since $1-f_{n_j}$ vanishes on ${\rm supp\,}\xi_0\subset\Lambda_{n_j}$. The other terms at the right are deterministic or depend only on $\xi_1$.
				We can conclude that
				\begin{align*}1&=\mathbb P^\theta
					\left(\xi_0(\Lambda)=\lim_{j\to\infty}
					\left(\mathbb E\int_\Lambda f_{n_j}\dx \xi-\int_\Lambda f_{n_j}\dx \xi_1\right)\right)
					\\&=\int_{\mathcal N(\Lambda)} 
					\mathbb P^\theta_{|{\mathbf v}}\left(\xi(\Lambda)=\lim_{j\to\infty}
					\left(\mathbb E\int_\Lambda f_{n_j}\dx \xi-\int_\Lambda f_{n_j}\dx \delta_{\mathbf v}\right)
					\right)\dx \mathbb P_1^\theta(\delta_{\mathbf v}),
				\end{align*}
				by Proposition \ref{prop:conditionalinf},
				and it follows that 
				\[\mathbb P^\theta_{|{\mathbf v}}\left(\xi(\Lambda)=\lim_{j\to\infty}
				\left(\mathbb E\int_\Lambda f_{n_j}\dx \xi-\int_\Lambda f_{n_j}\dx \delta_{\mathbf v}\right)=:\ell_{\mathbf v}
				\right)=1\quad\mbox{for $\mathbb P^\theta_1$-a.e.\ $\delta_{\mathbf v}$.} 
				\]
By Proposition \ref{prop:conditionalinf} (1), the map $\delta_{\mathbf v}\mapsto \mathbb P^\theta_{|\mathbf v}(\xi(\Lambda)=\ell)$ is $\mathcal C(\Lambda)$-measurable for any $\ell\in\mathbb N\cup\{0,\infty\}$, hence the preimage of $1$ under this map is in $\mathcal C(\Lambda)$.
But {by definition, this is the same as the preimage of $\ell$} under the map $\delta_{\mathbf v}\mapsto \ell_{\mathbf v}$, hence 
$\delta_{\mathbf v}\mapsto \ell_{\mathbf v}$ is $\mathcal C(\Lambda)$-measurable.
				We can conclude that
				$\mathbb P$ is marking rigid, and part (1) of the theorem is proved.
				
				\medskip
				
				For part (2), it follows from \cite[Theorem 4]{Soshnikov} that  
				$\mathbb P^\theta(\xi_0(\Lambda)<\infty)=0$ if $\Tr\,\mathrm M_{1-\theta}\mathrm K=\infty$, while
				$\mathbb P^\theta(\xi_0(\Lambda)<\infty)=1$ if $\Tr\,\mathrm M_{1-\theta}\mathrm K<\infty$.
				We then know that $\mathbb P$ is marking rigid from part (1), and it follows that, for any measurable $\theta:\Lambda\to[0,1]$ such that $\Tr\,\mathrm M_{1-\theta}\mathrm K<\infty$ and for $\mathbb P^\theta_1$-a.e.\ $\delta_{\mathbf v}$, there exists a finite number $\ell_{\mathbf v}$ such that $\mathbb P^\theta_{|\mathbf v}(\xi(\Lambda)=\ell_{\mathbf v})=1$. 
				By Proposition \ref{prop:conditionalinf} (2), $\mathbb P^\theta_{|\mathbf v}$ is a DPP induced by a Hermitian locally trace class operator, hence
				again by \cite[Theorem 4]{Soshnikov}, it is induced by an orthogonal projection $\mathrm K^\theta_{|\mathbf v}$ of rank $\ell_{\mathbf v}$.
}			\end{proof}
			\begin{remark}
		{		The above proof gives us more information about  $\xi_0(\Lambda)$. Depending on $\theta$, this number is either a.s.\ infinite, or a.s.\ equal to 
				\[\ell_{\mathbf v}=\lim_{j\to\infty}
%				\left(\mathbb E\int_\Lambda f_{n_j}(x)\dx \xi(x)-\int f_{n_j}\dx \delta_{\mathbf v}\right).
				{\mathbb E\int_\Lambda f_{n_j}\dx(\xi-\delta_\mathbf v).}\]
				It is surprising that this number does not depend explicitly on the marking function $\theta$. Of course, the {configurations $\delta_{\mathbf v}$ for which it holds} implicitly encode information about $\theta$.
		}	\end{remark}

			\medskip\ \medskip
			\section{OPEs on the real line or on the unit circle}\label{section:4}
				\subsection{OPEs on the real line}
			Let us consider the $N$-point OPE on the real line defined by \eqref{jpdf}.
			It is well-known that \eqref{jpdf} is a DPP on $(\mathbb R, w(x)\dx x)$, with kernel
			\begin{equation}\label{eq:CDkernelR}K_N(x,y)=\sum_{j=0}^{N-1}p_j(x)p_j(y),\end{equation}     where $p_j$ is the normalized orthogonal polynomial of degree $j$ with positive leading coefficient on the real line with respect to the weight $w(x)$.
			From the orthogonality of the polynomials, it follows that the integral operator $\mathrm K_N$ with kernel $K_N$ acting on $L^2(\mathbb R,
			w(x)\mathrm dx)$, defined by
			\begin{equation}
				\mathrm K_N f (y)=\int_{\mathbb R}K_N(x,y)f(y)w(y)\mathrm dy,
			\end{equation}
			is the orthogonal projection
			onto the $N$-dimensional space 
			\[H_N:=\{p: p \mbox{ is a polynomial of degree }\leq N-1\}.\]
			Alternatively, by the Christoffel-Darboux formula, we can write the correlation kernel in $2$-integrable form
			\begin{equation}\label{eq:CD}K_N(x,y)=\gamma_N\frac{p_N(x)p_{N-1}(y)-p_N(y)p_{N-1}(x)}{x-y},\end{equation}
			with $\gamma_N=\frac{\kappa_{N-1}}{\kappa_N}$, where $\kappa_n$ is the leading coefficient of $p_n$, or equivalently $\kappa_n^{-1}=\int_\mathbb Rp_n(x)x^nw(x)\mathrm dx$.
			See e.g.\ \cite{Deift} for more background and details about these ensembles.

			\subsection{OPEs on the unit circle}
			For general integrable weight functions $w$, \eqref{jpdfUC} is a DPP on the unit circle $\{z=e^{it}\}$ with respect to $w(e^{it})\dx t$, with correlation kernel 
			\begin{equation}\label{eq:CDkernelcircle}K_N(e^{i t},e^{i s})=\sum_{j=0}^{N-1}\varphi_j(e^{i t})\overline{\varphi_j(e^{is})},\end{equation}
where $\varphi_j$ is the normalized orthogonal polynomial of degree $j$ with positive leading coefficient on the unit circle with respect to the weight $w(e^{it})$.
			The associated integral operator $\mathrm K_N$ is the orthogonal projection onto the space 
			\[H_N:=\{\varphi: \varphi \mbox{ is a polynomial of degree }\leq N-1\}\]
			Alternatively, by the Christoffel-Darboux formula for orthogonal polynomials on the unit circle, we have the $2$-integrable form
			\[K_N(e^{it},e^{is})=\frac{e^{iN(t-s)}\varphi_N(e^{is})\overline{\varphi_N(e^{it})}-\varphi_N(e^{it})\overline{\varphi_N(e^{is})}}{1-e^{i(t-s)}}.\]
			For the uniform weight $w=1$, we have 
			\(\varphi_j(z)=(2\pi)^{-\frac 12}z^j\) and thus, after conjugation of the operator \(\mathrm K_N\), the kernel can be taken to be
			\[K_N(e^{it},e^{is})=\frac{1}{2\pi}\frac{\sin \frac{N(t-s)}{2}}{\sin\frac{t-s}{2}}.\]%e^{i\frac{N-1}2(\theta-\eta)}
			In the scaling limit  where $t-s=\frac{2\pi (u-v)}{N}$ and $N\to\infty$, \(\frac {2\pi}NK_N(e^{it},e^{is})\) converges to the sine kernel
			\begin{equation}\label{sine kernel}K^{\sin}(u,v)=\frac{\sin\pi(u-v)}{\pi(u-v)},\end{equation} uniformly for $u,v$ in compact subsets of the real line.
			See e.g.\ \cite{Forrester, Mehta} for details.

			\subsection{Conditional ensembles associated to OPEs}
		From Proposition \ref{prop:proj}, it follows that the conditional ensemble $\mathbb P^\theta_{|\mathbf v}$ is the DPP, on $(\mathbb R, (1-\theta(x))w(x)\dx x)$ or on the unit circle with measure $(1-\theta(e^{it}))w(e^{it})\dx t$, with kernel $\widetilde K_N$ of the orthogonal projection onto the space
			\[H_{N,\mathbf v}:=\{p : \mbox{$p$ is a polynomial of degree $\leq N-1$, and $p(v)=0\ \forall v\in\mathbf v$}\}.\]

Let us now define \[w^\theta_{|\mathbf v}(x):=(1-\theta(x))w(x)\prod_{v\in\mathbf v}|x-v|^2\]
			for the real line, and 
			\[w^\theta_{|\mathbf v}(e^{it}):=(1-\theta(e^{it}))w(e^{it})\prod_{v\in\mathbf v}|e^{it}-v|^2\]
			for the unit circle. In the case of the real line, we then have $\widetilde K_N(x,y)=\prod_{v\in\mathbf v}(x-v)(y-v)K_n(x,y)$, with $n:=N-\#\mathbf v$ and with
			$K_n$ the Christoffel-Darboux kernel \eqref{eq:CDkernelR} with $N$ replaced by $n$ and $w$ by $w^\theta_{|\mathbf v}$.
	It follows that we can also 
		see $\mathbb P^\theta_{|\mathbf v}$ as a DPP on $(\mathbb R, w^{\theta}_{|\mathbf v}(x)\dx x)$ with kernel $K_n$,
which implies that it is the $n$-point OPE
\[
			\frac{1}{Z_n}\Delta(\mathbf x)^2\ \prod_{j=1}^n w^\theta_{|\mathbf v}(x_j)\mathrm dx_j,\qquad \Delta(\mathbf x)=\prod_{1\leq i<j\leq N}(x_j-x_i).
		\] 
In the case of the unit circle, we obtain similarly that $\mathbb P^\theta_{|\mathbf v}$ is the $n$-point OPE 
\[
			\frac{1}{Z_n}|\Delta(\mathbf {e^{it}})|^2\ \prod_{j=1}^n w^\theta_{|\mathbf v}(e^{it_j})\mathrm dt_j,\qquad \Delta(\mathbf{e^{it}})=\prod_{1\leq l<k\leq N}(e^{it_k}-e^{it_j}),\qquad t_j\in[0,2\pi).
		\] 
Summarizing the above, we have proved the following result.
\begin{proposition}	\label{prop:OPE}	
		If $\mathbb P$ is the $N$-point OPE with weight $w$ on the real line or the unit circle and $n:=N-\#\mathbf v>0$, then $\mathbb P^\theta_{|\mathbf v}$ is the $n$-point OPE with weight $w^\theta_{|\mathbf v}$ on the real line or the unit circle.
\end{proposition}

			\subsection{Unitary invariant ensembles and scaling limits}
			The above form of the conditional ensembles has the remarkable consequence that any unitary invariant ensemble \eqref{jpdf} with $V(x)\geq x^2$, can be constructed theoretically from the GUE: to see this, consider the conditional ensemble  $\mathbb P^\theta_{|\emptyset}$ with $w(x)=e^{-Nx^2}$ the Gaussian weight in \eqref{jpdf}, and with
			$\theta(x)=1-e^{-N(V(x)-x^2)}\in[0,1]$. The latter has the joint probability distribution \eqref{jpdf}, but now with weight
			\[w^\theta_{|\emptyset}(x):=(1-\theta(x))e^{-Nx^2}=e^{-NV(x)}.\]
This is of course not of any practical use for $N$ large, because the event  on which we condition then has very small probability unless $V(x)$ is close to $x^2$ (note that there exist algorithms to generate DPPs in general, see e.g.\ \cite{HoughKrishnapurPeresVirag}).
			Nevertheless, it is striking that the GUE encodes any of the above unitary invariant ensembles via marking and conditioning.
			
			\medskip
			
			This becomes even more surprising if we look at
			scaling limits of the correlation kernels. It is a classical fact that the GUE converges to the sine point process in the bulk scaling limit and to the Airy point process in the edge scaling. It is also understood that conditioning on an eigenvalue and scaling around this eigenvalue leads to the bulk Bessel point process, and that conditioning on a gap leads to the hard edge Bessel kernel.
			But unitary invariant ensembles admit for special choices of $V$ also more complicated limit processes, associated to Painlev\'e equations and hierarchies \cite{Duits}.
			In fact, it follows from the above that these Painlev\'e point processes are already encoded in the GUE eigenvalue distribution, if one combines a suitable conditioning with taking scaling limits.
			
			\subsection{Marginal distribution of mark $0$ points with known number of mark $1$ points.}

			The construction of the conditional ensembles in Section \ref{section:2} passed through the marked point process conditioned on having $m$ mark $1$ particles. In this case, we can make these ensembles more explicit.
			Indeed, from \eqref{jpdf}, we obtain that the marginal distribution  of the mark $0$ particles, conditioned on having exactly $m$ mark $1$ particles, is  given by 
			\[
				\frac{1}{Z_{N,m}}|\Delta(\mathbf u)|^2\ \left(\int_{\mathbb R^m}|\Delta(\mathbf v)|^2
				\ \prod_{k=1}^m\left(\prod_{\ell=1}^n|u_\ell-v_k|^2\right)\theta(v_k)w(v_k)\dx v_k\right)\ \prod_{j=1}^n(1-\theta(u_j))w(u_j)\dx u_j,  
			\] 
			where 
			\[
				Z_{N,m}=\int_{\mathbb R^n}|\Delta(\mathbf u)|^2\ \left(\int_{\mathbb R^m}|\Delta(\mathbf v)|^2
				\ \prod_{k=1}^m\left(\prod_{\ell=1}^n|u_\ell-v_k|^2\right)\theta(v_k)w(v_k)\dx v_k\right)\ \prod_{j=1}^n(1-\theta(u_j))w(u_j)\dx u_j.  
			\] 
			By Heine's formula, the $v$-integral can be written as a Hankel determinant in the case of the real line, and as a Toeplitz determinant in the case of the unit circle.
			For the real line, 
			defining the Hankel determinant as
			\[H_{m}(f)=\det\left(f_{j+k}\right)_{j,k=0}^{m-1},\qquad f_\ell=\int_{\mathbb R} x^\ell f(x)\dx x,\] we have 
			\begin{equation}
				\frac{1}{Z_{N,m}'}|\Delta(\mathbf u)|^2\ H_{m}\left(\theta.w.\prod_{j=1}^n(.-u_j)^2\right)\ \prod_{j=1}^n(1-\theta(u_j))w(u_j)\dx u_j,
			\end{equation} 
			with $Z_{N,m}'={\int_{\mathbb R^n}|\Delta(\mathbf u)|^2\ H_{m}\left(\theta.w.\prod_{j=1}^n(.-u_j)^2\right)\ \prod_{j=1}^n(1-\theta(u_j))w(u_j)\dx u_j}$.
			For the unit circle, 
			defining the Toeplitz determinant as
			\[T_{m}(g)=\det\left(g_{j-k}\right)_{j,k=0}^{m-1},\qquad g_\ell=\frac{1 }{2\pi}\int_{0}^{2\pi} e^{-i\ell t} g(e^{i t})\dx t,\]
			we obtain 
			\begin{equation}
				\frac{1}{Z_{N,m}'}|\Delta(\mathbf{e^{it}})|^2\ T_{m}\left(\theta.w.\prod_{j=1}^n|.-e^{it_j}|^2\right)\ \prod_{j=1}^n(1-\theta(e^{it_j}))w(e^{it_j})\dx t_j,
			\end{equation} 
			with {$Z_{N,m}'={\int_{(0,2\pi)^n}|\Delta(\mathbf{e^{it}})|^2\ T_{m}\left(\theta.w.\prod_{j=1}^n|.-e^{it_j}|^2\right)\ \prod_{j=1}^n(1-\theta(e^{it_j}))w(e^{it_j})\dx t_j}$}.
			Similar formulas hold for the marginal distributions of the mark $1$ points. {Alternatively, by \cite[Theorem 3.2]{BaikDeiftStrahov}, we can write both densities, with either \(x_j=u_j,\ \dx\mu(x_j)=\dx x_j\) or \(x_j=e^{it_j},\ \dx\mu(x_j)=\dx t_j\), as
			\begin{equation*}
				\frac 1{Z^{''}_{N,m}}\det\begin{pmatrix}
					K_{N}^{\theta}(x_l,x_k)
				\end{pmatrix}_{l,k=1}^n\prod_{j=1}^n(1-\theta(x_j))w(x_j)\dx\mu(x_j),
		\end{equation*}		
		with a new normalization constant \(Z_{N,m}^{''}\) obtained in a similar manner, and where \(K_{N}^{\theta}(x,y)\) is the kernel inducing the point processes \eqref{jpdf}/\eqref{jpdfUC} with \(N=n+m\) particles and weight function \(\theta w\).} There is no reason to believe that these marginal distributions are in general DPPs, but they do have a special Hankel or Toeplitz determinant structure. In particular, probabilities can be expressed in terms of integrals of Toeplitz or Hankel determinants, which can in some cases be computed asymptotically as $m\to\infty$.
			Similar integrals of Toeplitz and Hankel determinants appear in the study of moments of moments in random matrix ensembles, connected to the study of extreme values of characteristic polynomials, see e.g.\ \cite{BaileyKeating, FyodorovKeating}.

			\medskip\ \medskip 
			\section{Integrable DPPs}\label{section:5}
			In this section, we will consider DPPs $\mathbb P$ on curves $\Lambda$ in the complex plane with $k$-integrable kernels of the form \eqref{k-intkernel}. For simplicity, let us assume that $\Lambda$ is a smooth closed curve on $\mathbb C\cup\{\infty\}$ without self-intersections, that the functions $f_1,\ldots, f_k$ and $g_1,\ldots, g_k$ are smooth functions on $\Lambda$, and that the reference measure is smooth with respect to $\dx z$, i.e.\ $\dx\mu(z)=h(z)\dx z$ with $h$ smooth (say $C^\infty$, even if one can proceed with less regularity if needed) on $\Lambda$. Even if $\dx z$ is not a positive measure on $\Lambda$, by mapping $K(x,y)$ to {$K(x,y)h(y)$}, we can then work with a kernel
			\[K(x,y)=\frac{\mathbf f(x)^T\mathbf g(y)}{x-y}=\frac{\mathbf g(y)^T\mathbf f(x)}{x-y},\] 
			with column vectors \(\mathbf f=(f_j)_{j=1}^k, \mathbf g=(g_j)_{j=1}^k\),
			with respect to the complex measure $\dx z$,
and with the associated integral operator $\mathrm K$ acting on $L^2(\Lambda,\dx z)$.	

\medskip
\subsection{General integrable kernels}
Let us first show that the Palm kernels $K_{\mathbf v}$ are also of $k$-integrable form.\begin{proposition}	\label{prop:Palm}		
For any ${\mathbf v}=\{v_1,\ldots, v_m\}$ such that $\det K(\mathbf v, \mathbf v)>0$, the kernel of the reduced Palm measure $\mathbb P_{\mathbf v}$	 is of $k$-integrable form		
			$K_{\mathbf v}(x,y)=\frac{{\mathbf f}_{\mathbf v}(x)^T\mathbf g_{\mathbf v}(y)}{x-y}$ , and the $j$-th entries of ${\mathbf f}_{\mathbf v}$ and ${\mathbf g}_{\mathbf v}$ are given by
\[f_{{\mathbf v},j}(x)=\frac{1}{\det K(\mathbf v,\mathbf v)}\det\begin{pmatrix}f_j(x)&K(x,\mathbf v)\\
f_j(\mathbf v)&K(\mathbf v, \mathbf v)\end{pmatrix},\quad 
g_{{\mathbf v},j}(y)=\frac{1}{\det K(\mathbf v,\mathbf v)}\det\begin{pmatrix}g_j(y)&g_j(\mathbf v)\\
K(\mathbf v,y)&K(\mathbf v, \mathbf v)\end{pmatrix},
\]		
where
$K(\mathbf v, \mathbf v)$ represents the $m\times m$ matrix with $(i,j)$-entry $K(v_i,v_j)$, $K(x,\mathbf v), g_j(\mathbf v)$ represent $m$-dimensional row vectors with entries $K(x,v_\ell), g_j(v_\ell)$, and $K(\mathbf v,y), f_j(\mathbf v)$ represent $m$-dimensional column vectors with entries $K(v_i,y), f_j(v_i)$.						
\end{proposition}
\begin{proof}
Using the block determinant formula \eqref{eq:block}, we have that ${f}_{\mathbf v,j}$ as defined in the statement of the proposition is given by
\[{f}_{\mathbf v,j}(x)=f_j(x)-K(x,\mathbf v)K(\mathbf v, \mathbf v)^{-1} f_j(\mathbf v).\]
Now let \(\mathbf v=\mathbf v'\sqcup\{v\}\) and assume without loss of generality that \(\det K(\mathbf v',\mathbf v')>0\), then using again the block determinant formula \eqref{eq:block} one obtains 
\begin{equation*}
	\begin{aligned}		
		f_{\mathbf v,j}(x)&=\frac 1{\det K(\mathbf v,\mathbf v)}\det\begin{pmatrix}
			
			f_j(x) & K(x,v) & K(x,\mathbf v') \\ f_j(v) & K(v,v) &  K(v,\mathbf v')\\ f_j(\mathbf v') & K(\mathbf v',v) & K(\mathbf v',\mathbf v')
		\end{pmatrix}\\
	&=\frac{\det K(\mathbf v',\mathbf v')}{\det K(\mathbf v,\mathbf v)}\det\left(\begin{pmatrix} f_j(x) & K(x,v) \\ f_j(v) & K(v,v) \end{pmatrix}
	-\begin{pmatrix}
		K(x,\mathbf v')\\ K(v,\mathbf v')
	\end{pmatrix}K(\mathbf v',\mathbf v')^{-1}\begin{pmatrix}f_j(\mathbf v') & K(\mathbf v',v)
\end{pmatrix}
	\right)\\
	&=\frac 1{K_{\mathbf v'}(v,v)}\det \begin{pmatrix}
		f_{\mathbf v',j}(x) & K_{\mathbf v'}(x,v) \\ f_{\mathbf v',j}(v) & K_{\mathbf v'}(v,v)
	\end{pmatrix},
	\end{aligned}
\end{equation*} 
which implies that \(\mathbf f_\mathbf v=(\mathbf f_{\mathbf v'})_v\), and similarly for \(\mathbf g_\mathbf v\). Since also \(\mathrm K_\mathbf v=(\mathrm K_{\mathbf v'})_v\), it now suffices to prove the result for $\mathbf v=\{v\}$. We then easily verify by \eqref{eq:Kv} that
\begin{equation*}
	\begin{aligned}
		(x-y)K_v(x,y)&=\mathbf f(x)^T\mathbf g(y)-((x-v)+(v-y))\frac{K(x,v)K(v,y)}{K(v,v)}\\
		&=\mathbf f(x)^T\mathbf g(y)-\mathbf f(x)^T\mathbf g(v)\frac{K(v,y)}{K(v,v)}-\frac{K(x,v)}{K(v,v)}\mathbf f(v)^T\mathbf g(y)\\
		&=\left(\mathbf f(x)-\frac{K(x,v)}{K(v,v)}\mathbf f(v)\right)^T\left(\mathbf g(y)-\frac{K(v,y)}{K(v,v)}\mathbf g(v)\right),
	\end{aligned}
\end{equation*}
since \(\mathbf f^T\mathbf g=0\). To complete the proof, it remains to check that \(\mathbf f_v^T\mathbf g_v=0\), but this follows from a similar computation: 
\begin{equation*}
	\begin{aligned}
		\mathbf f_v(x)^T\mathbf g_v(x)&=-\mathbf f(x)^T\mathbf g(v)\frac{K(v,x)}{K(v,v)}-\frac{K(x,v)}{K(v,v)}\mathbf f(v)^T\mathbf g(x)=K(x,v)\frac{\mathbf f(v)^T\mathbf g(x)}{K(v,v)}-\frac{K(x,v)}{K(v,v)}\mathbf f(v)^T\mathbf g(x)=0.
	\end{aligned}
\end{equation*}
\end{proof}

\medskip

Next, we will explain how the kernel of the point process $\mathbb P^\theta_{|\mathbf v}$ on $(\Lambda,\mu_0^\theta)$, with kernel of the operator $\mathrm K(1-\mathrm M_\theta\mathrm K)^{-1}$, can be characterized in terms of a RH problem. For this, we rely on the IIKS method developed in \cite{IIKS, DIZ}.

\medskip

In what follows, we assume that the entries of $\sqrt{\theta}\mathbf g$ and $\sqrt{\theta}\mathbf f$ are smooth, bounded and integrable functions on $\Lambda$ which decay as $z\to\infty$, and that their derivatives are also bounded and integrable.

\medskip

			Let us consider the following RH problem.
			\subsubsection*{RH problem for $Y$}
			\begin{itemize}
				\item[(a)]$Y:\mathbb C\setminus\Lambda\to\mathbb C^{k\times k}$ is analytic; we mean by this that every entry of the matrix is analytic in $\mathbb C\setminus\Lambda$.
				\item[(b)] $Y$ has continuous boundary values $Y_\pm$ when $\Lambda$ is approached from the left ($+$) or right ($-$), with respect to the orientation chosen for $\Lambda$, and they are related by
				\begin{equation}
					Y_+(z)=Y_-(z)J_Y(z),\qquad J_Y(z)=I_k-2\pi i \theta(z)\mathbf f_{\mathbf v}(z)\mathbf  g_{\mathbf v}(z)^T,\qquad z\in\Lambda,
				\end{equation}
				where $I_k$ is the $k\times k$ identity matrix.
				\item[(c)]As $z\to\infty$, $Y(z)\to I_k$ uniformly.
			\end{itemize}

\medskip			
		The following is a consequence of results from, e.g., \cite[Section 2]{DIZ}, see also \cite{IIKS} and \cite{BertolaCafasso}.
			\begin{proposition}
Suppose that $\mathrm M_{\sqrt{\theta}+1_B} \mathrm K_{\mathbf v}\mathrm M_{\sqrt{\theta}+1_B}$ is trace class on $L^2(\Lambda,\dx z)$ for any bounded Borel set $B$, and that
			 $\det(1-\mathrm M_{\sqrt{\theta}}\mathrm K_{\mathbf v}\mathrm M_{\sqrt{\theta}})\neq 0$. 
				\begin{enumerate}\item The RH problem for $Y$ is uniquely solvable, and the solution $Y(z)$ is invertible for any $z\in\mathbb C\setminus\Lambda$.
					\item The DPP $\mathbb P^\theta_{|\mathbf v}$ on $(\Lambda,(1-\theta)\dx z)$ is characterized by the $k$-integrable kernel
					\begin{equation}
						K^\theta_{|\mathbf v}(x,y)=\frac{\mathbf f^{\theta}_{|\mathbf v}(x)^T\mathbf g^{\theta}_{|\mathbf v}(y)}{x-y}%\frac{\sum_{j=1}^k{F_j}(x)G_j(y)}{x-y}=,
					\end{equation}
where					
\begin{equation}
						\mathbf f^{\theta}_{|\mathbf v}=Y_\pm{\bf f}_{\mathbf v},\qquad \mathbf g^{\theta}_{|\mathbf v}=Y_\pm^{-T}\mathbf g_{\mathbf v},
					\end{equation}
and the above expressions are independent of the choice $\pm$ of boundary value, with $Y_{\pm}^{-T}$ denoting the inverse transpose of the matrix $Y_\pm$.				
Consequently
					\begin{equation}\label{eq:IIKSkernel}
						K^\theta_{|\mathbf v}(x,y)=\frac{1}{x-y}{\bf g}_{\mathbf v}(y)^TY_\pm(y)^{-1}Y_\pm(x){\bf f}_{\mathbf v}(x).
					\end{equation}
				\end{enumerate}
			\end{proposition}
\begin{proof}
Observe first that, because of the assumptions and Proposition \ref{prop:Palm},
$\theta(x) \mathbf f_\mathbf v(x)\mathbf g_\mathbf v(x)^T$ is also smooth, in $L^2(\Lambda{, \dx z})$, and decaying as $x\to\infty$, $x\in\Lambda$.
We then set $\mathrm A=\mathrm M_{\sqrt\theta}\mathrm K_{\mathbf v}\mathrm M_{\sqrt\theta}$ and $V=I_k-2\pi i\theta\mathbf f_{\mathbf v}\mathbf g_{\mathbf v}^T$, and apply \cite[Lemma 2.12]{DIZ}: this result states that
\[(1-\mathrm A)^{-1}-1=\mathrm A(1-\mathrm A)^{-1}=\mathrm M_{\sqrt\theta}\mathrm K_{\mathbf v}\left(1-\mathrm M_\theta\mathrm K_\mathbf v\right)^{-1}\mathrm M_{\sqrt\theta}\]
has kernel
$\frac{{\bf F}(x)^T{\bf G}(y)}{x-y}$
with 
\[{\bf F}=Y_\pm{\sqrt{\theta}\bf f}_\mathbf v,\qquad {\bf G}=Y_\pm^{-T}\sqrt{\theta}\mathbf g_\mathbf v.\]
Hence, if $\theta$ has no zeros on $\Lambda$, the operator 
$\mathrm K_\mathbf v(1-\mathrm M_\theta \mathrm K_\mathbf v)^{-1}$
has kernel 
$\frac{\mathbf f^\theta_{|\mathbf v}(x)^T\mathbf g^\theta_{|\mathbf v}(y)}{x-y}$ on $L^2(\Lambda,\dx z)$
with 
\[\mathbf f^\theta_{|\mathbf v}=\frac 1{\sqrt\theta}\mathbf F,\qquad \mathbf g^\theta_{|\mathbf v}=\frac 1{\sqrt\theta}\mathbf G,\]
and the result follows from Theorem \ref{theorem:condnonempty}.
If $\theta$ has zeros on $\Lambda$, the result does not directly follow, but it is readily seen that one can follow the proof of \cite[Lemma 2.12]{DIZ} to prove the result also in this case. 
\end{proof}		
\begin{remark}
The smoothness and decay of the entries of $\theta\mathbf f\mathbf g^T$ are assumptions we make to avoid technical complications, and which guarantee smooth boundary values $Y_\pm$ and uniform convergence at infinity. One can also proceed with less regularity, but then care must be taken about the sense of the boundary values of $Y$, which are not necessarily continuous, and about the convergence at infinity, which is not necessarily uniform, see e.g.\ \cite{Deift, DIZ} for general theory of RH problems.
\end{remark}	

The above results imply that given $K, \theta, \mathbf v$, we obtain $K^\theta_{|\mathbf v}$ by first computing $K_{\mathbf v}$, and then solving the RH problem for $Y$.
Next, we explain how to bypass this procedure by characterizing $K^\theta_{|\mathbf v}$ directly in terms of a RH problem which depends in a simple explicit way on $K,\theta,\mathbf v$, without need to go through the transformation $K\mapsto K_{\mathbf v}$.
For that purpose, let us construct a rational matrix-valued function $R$, which will allow us to connect $\mathbf f,\mathbf g$ with $\mathbf f_{\mathbf v},\mathbf g_{\mathbf v}$.

For a singleton \(\mathbf v=\{v\}\), we observe that
\begin{equation*}
	\mathbf f_v(x)=\mathbf f(x)-\mathbf f(v)\frac{K(x,v)}{K(v,v)}=\left(I_k-\frac{R_1}{x-v}\right)\mathbf f(x),\qquad R_1=\frac{\mathbf f(v)\mathbf g(v)^T}{K(v,v)},
\end{equation*}
and similarly since \(R_1^2=0\),
\begin{equation*}
	\mathbf g_v(x)^T=\mathbf g(x)^T\left(I_k+\frac{R_1}{x-v}\right)=\mathbf g(x)^T\left(I_k-\frac{R_1}{x-v}\right)^{-1}.
\end{equation*}
For the general case \(\mathbf v=\{v_1,...,v_m\}\), we inductively define the matrices \(R_j\) for \(j=1,...,m\) by
\[R_j=\frac{\mathbf f_{v_1,\ldots, v_{j-1}}(v_j)\mathbf g_{v_1,\ldots, v_{j-1}}(v_j)^T}{K_{v_1,\ldots, v_{j-1}}(v_j,v_j)},\]
satisfying \(R^2_j=0\), and
\begin{equation}
	\label{def:R}
	R(z)=\left(I_k+\frac{R_1}{z-v_1}\right)\left(I_k+\frac{R_2}{z-v_2}\right)\cdots\left(I_k+\frac{R_m}{z-v_m}\right),\qquad z\in\mathbb C\setminus\{v_1,\ldots, v_m\}.
\end{equation}
Then $R$ has determinant identically equal to $1$, and
\[\mathbf f_{\mathbf v}(x)=R(x)^{-1}\mathbf f(x),\qquad \mathbf g_{\mathbf v}(x)^T=\mathbf g(x)^TR(x),\]
so that we can rewrite the jump matrix $J_Y$ as 
	\begin{equation}\label{eq:transformJY}J_Y(x)=I_k-2\pi i \theta(x) \mathbf f_{\mathbf v}(x)\mathbf g_{\mathbf v}(x)^T=R(x)^{-1}	\left(I_k-2\pi i \theta(x) \mathbf f(x)\mathbf g(x)^T\right)
R(x).	
	\end{equation}	
Note that although the construction of \(R\) uses a certain order of \(v_1,...,v_m\), the result only depends on the unordered set \(\mathbf v\), as it can be checked that
\begin{equation*}
	R(z)=I_k+\frac{\mathbf f(\mathbf v)}{z-\mathbf v} K(\mathbf v,\mathbf v)^{-T}\mathbf g(\mathbf v)^T=\left(I_k-\mathbf f(\mathbf v)K(\mathbf v,\mathbf v)^{-T}\left(\frac{\mathbf g(\mathbf v)}{z-\mathbf v}\right)^T \right)^{-1},
\end{equation*}	
where \(\frac{\mathbf f(\mathbf v)}{z-\mathbf v}\) and \(\mathbf g(\mathbf v)\) are the \(k\times m\) matrices whose \(i\)-th columns are \(\frac{\mathbf f(v_i)}{z-v_i}\) and \(\mathbf g(v_i)\), and similarly for the others. It turns out that \(R\) is a rational function which can also be characterized by a discrete RH problem (in fact, $R^{-1}$ is the solution to the below RH problem for $U$ with \(\theta=0\)).
			Let us now define
			\begin{equation}\label{def:T}
			U(z)=Y(z)R(z)^{-1},
			\end{equation}
then $U$ satisfies the following RH problem.

\subsubsection*{RH problem for $U$}
				\begin{enumerate}
					\item Each entry of \(U:\mathbb C\setminus\Lambda\to\mathbb C^{k\times k}\) is analytic.
					\item On \(\Lambda\setminus\mathbf v\), $U$ has continuous boundary values \(U_\pm\) which satisfy the jump condition 
					\begin{equation*}
						U_+=U_-(I_k-2\pi i \theta \mathbf f \mathbf g^T),
					\end{equation*}
					while for each \(v\in\mathbf v\), the residue \(\rho_U(v)=\lim_{z\to v}(z-v)U(z)\) is well-defined and given by
					\begin{equation*}
						\rho_U(v)=-\lim_{z\to v}U(z)\frac{\mathbf f(v)\mathbf g(v)^T}{K(v,v)}.
					\end{equation*}
					\item As \(z\to\infty\), $U(z)\to I_k$ uniformly.

\end{enumerate}

Conditions (1) and (3) are immediately verified. To check the jump relation (2) for $U$, it suffices to use \eqref{def:T} and \eqref{eq:transformJY}. For the residues of \(U\), observe that it is sufficient to verify the condition for \(v=v_1\) by the iterative construction of \(R\). We then have by \eqref{def:R}, \eqref{def:T}, \(R^2_1=0\), and the fact that \(Y_+(v_1)=Y_-(v_1)\) (since \(\mathbf f_\mathbf v(v_1)=\mathbf g_\mathbf v(v_1)=0\)) that 		
\begin{align*}
					\lim_{z\to v_1}(z-v_1)U(z)&=-Y_\pm(v_1){\rm Res}_{z=v_1}\,R^{-1}=-Y_\pm(v_1)
\left(I_k-\frac{R_m}{v_1-v_m}\right)\cdots\left(I_k-\frac{R_2}{v_1-v_2}\right)R_1\\
					&=-\lim_{z\to v_1}Y(z)\left(I_k-\frac{R_m}{z-v_m}\right)\cdots\left(I_k-\frac{R_1}{z-v_1}\right)R_1
	=-\lim_{z\to v}U(z)\frac{\mathbf f(v)\mathbf g(v)^T}{K(v,v)}.
						\end{align*}
In conclusion, we have the following result.

		\begin{proposition}	
Suppose that $\mathrm M_{\sqrt{\theta}+1_B} \mathrm K_{\mathbf v}\mathrm M_{\sqrt{\theta}+1_B}$ is trace class on $L^2(\Lambda,\dx z)$ for any bounded Borel set $B$, and that
			 $\det(1-\mathrm M_{\sqrt{\theta}}\mathrm K_{\mathbf v}\mathrm M_{\sqrt{\theta}})\neq 0$.			There exists a unique solution \(U\) to the RH problem for $U$ which is furthermore invertible and satisfies
				\begin{equation*}
					\mathbf f^\theta_{|\mathbf v}=U_\pm \mathbf f,\quad\quad\quad \mathbf g^\theta_{|\mathbf v}=U^{-T}_\pm \mathbf g,
				\end{equation*}
				and
				\[
K^{\theta}_{|\mathbf v}(x,y)=\frac{1}{x-y}					\mathbf g(y)^TU_\pm(y)^{-1}U_\pm(x)\mathbf f(x).				
				\]
			\end{proposition}

			\subsection{Integrable kernels characterized by a RH problem}
			The above RH characterization of $K^\theta_{|\emptyset}$ and $K^\theta_{|\mathbf v}$ is particularly useful in cases where the kernel $K$ of the DPP $\mathbb P$ itself can also be characterized in terms of a RH problem. In such a case, the IIKS method allows to transform the RH problem to an {\em undressed} RH problem which is in a form amenable to asymptotic analysis and to derive differential equations \cite{BertolaCafasso, BertolaCafasso1, IIKS}.
				
\medskip			
		
		Such a RH characterization is available for many important 2-integrable DPPs, like OPEs and the DPPs characterized by the Airy kernel, the sine kernel, the Bessel kernels, the confluent hypergeometric kernels, and kernels connected to Painlev\'e equations. Multiple orthogonal polynomial ensembles \cite{KuijlaarsMOP} and their scaling limits like Pearcey and tacnode kernels are examples of $k$-integrable kernels with $k>2$, which can also be characterized through a ($k\times k$) RH problem.
		
		\medskip
		
		Let us illustrate this in the case $k=2$. 		
			\medskip
			
			Suppose that we can write 
			\begin{equation}\label{fg Psi}
K(x,y)=\frac{\mathbf f(x)^T\mathbf g(y)}{x-y},\qquad 				\mathbf f(x)=\frac{w(x)}{2\pi i}\Psi_\pm(x)\begin{pmatrix}1\\0\end{pmatrix},\qquad \mathbf g(y)^T=\begin{pmatrix}0&1\end{pmatrix}\Psi_\pm(y)^{-1},
			\end{equation}
			for a smooth bounded function $w:\Lambda\to[0,+\infty)$, where $\Psi$ satisfies a RH problem of the following form.

			\subsubsection*{RH problem for $\Psi$}
			\begin{enumerate}
				\item \(\Psi : \mathbb C\setminus\Lambda\to\mathbb C^{2\times 2}\) is analytic.
				\item $\Psi$ has continuous  boundary values $\Psi_\pm$, and they are related by
				\begin{equation*}
					\Psi_+(z)=\Psi_-(z)\begin{pmatrix}1&w(z)\\0&1\end{pmatrix},\qquad z\in\Lambda,
				\end{equation*}
				for some smooth bounded function $w:\Lambda\to\mathbb C$.
				\item For some \(\Psi_\infty : \mathbb C\setminus\Lambda\to\mathbb C^{2\times 2}\) such that \(\det\Psi_\infty(z)=1\), we have	\begin{equation*}
					\Psi(z)=\left(I_2+O(z^{-1})\right)\Psi_\infty(z),
				\end{equation*}
			uniformly as $z\to\infty$.
			\end{enumerate}

				Then, it is straightforward to show that $\det\Psi(z)\equiv 1$, hence $\Psi(z)$ is an invertible matrix for every $z\in\mathbb C\setminus\Lambda$, and that there is only one solution to the RH problem for $\Psi$.
	
			The third RH condition would be trivially valid with $\Psi_\infty=\Psi$,		
but as we illustrate in examples below, one usually prefers to specify a simpler explicit function $\Psi_\infty$				to describe the asymptotic behavior of $\Psi$, in order to facilitate further analysis of the RH problem.
Observe that the RH conditions imply that the first column of $\Psi$ and the second row of $\Psi^{-1}$ extend to entire functions in the complex plane, and hence that $\mathbf f/w$ and $\mathbf g$ extend to entire functions.

Let us define
\begin{equation}
\Psi^\theta_{|\mathbf v}=U\Psi.\end{equation} Then, $\Psi^\theta_{|\mathbf v}$ is invertible and it is the unique solution to the following RH problem.
			\subsubsection*{RH problem for $\Psi^\theta_{|\mathbf v}$}
			\begin{enumerate}
				\item Each entry of \(\Psi^\theta_{|\mathbf v} : \mathbb C\setminus\Lambda\to\mathbb C^{2\times 2}\) is analytic.
				\item $\Psi^\theta_{|\mathbf v}$ has continuous boundary values $\Psi^\theta_{|\mathbf v\pm}$ on $\Lambda\setminus\mathbf v$ and they are related by			
				\begin{equation*}
					\Psi^\theta_{|\mathbf v+}(z)=\Psi_{|\mathbf v-}^\theta(z)\begin{pmatrix}1& w(z)(1-\theta(z))\\0&1\end{pmatrix},
				\end{equation*}
				while as \(z\to v\in\mathbf v\) we have
				\begin{equation*}
					\Psi^\theta_{|\mathbf v}(z)=\mathcal O(1)(z-v)^{\sigma_3},\qquad \sigma_3:=\begin{pmatrix}1&0\\0&-1\end{pmatrix}.
				\end{equation*}
				\item As $z\to\infty$, we have the uniform asymptotics
				\[\Psi^\theta_{|\mathbf v}(z)=\left(I_2+\mathcal O(z^{-1})\right)\Psi_\infty(z).\]
				
			\end{enumerate}
			
The first and the third conditions are immediate from the corresponding ones for $U$ and $\Psi$. The jump relation is obtained from the jump relation for $U$ and the one for $\Psi$ along with \eqref{fg Psi}:
\begin{equation*}
	U_+\Psi_+=U_-\left(I_2-2\pi i\theta \frac{w}{2\pi i}\Psi_-\begin{pmatrix}1 \\ 0\end{pmatrix}\begin{pmatrix}
	 0 & 1
	\end{pmatrix}\Psi_+^{-1}\right)\Psi_+=U_-\Psi_-\left(\begin{pmatrix}
	1 & w \\ 0 &1
\end{pmatrix}-w\theta\begin{pmatrix}
0 & 1 \\ 0& 0
\end{pmatrix}\right)	.
\end{equation*}
The singular behavior near \(\mathbf v\) is obtained in a similar manner: the second column of \(\Psi^\theta_{|\mathbf v}(z-v)^{-\sigma_3}\) is obviously \(\mathcal O(1)\) since the second column of \(U\) is $\mathcal O((z-v)^{-1})$ as $z\to v\in\mathbf v$, while for the first column we notice that for each \(v\in\mathbf v\), by \eqref{fg Psi},{
\begin{equation*}
	\begin{aligned}
		\lim_{z\notin\Lambda\to v}\Psi^\theta_{|\mathbf v}(z)\begin{pmatrix}
			1 \\0
		\end{pmatrix}&=\lim_{z\in\Lambda\to v}\Psi^\theta_{|\mathbf v\pm}(z)\begin{pmatrix}
			1 \\0
		\end{pmatrix}=Y_\pm(v)\lim_{z\in \Lambda\to v}R(z)^{-1}\Psi_\pm(z)\begin{pmatrix}
			1 \\0
		\end{pmatrix}\\
	&=Y_\pm(v)\frac{2\pi i}{w(v)}\lim_{z\in \Lambda\to v}R(z)^{-1}\frac{w(z)}{2\pi i}\Psi_\pm(z)\begin{pmatrix}
		1 \\0
	\end{pmatrix}=Y_\pm(v)\frac{2\pi i}{w(v)}\mathbf f_\mathbf v(v)=0.
	\end{aligned}
\end{equation*}}			
Moreover, we have by \eqref{eq:IIKSkernel} that the kernel of the conditional ensemble is given by
			\begin{align}
				\nonumber K^\theta_{|\mathbf v}(x,y)&=\frac{1}{x-y}\begin{pmatrix}0&1\end{pmatrix}\Psi_{\pm}(y)^{-1}U_\pm(y)^{-1}
				U_\pm(x)\Psi_\pm(x)\begin{pmatrix}1\\0\end{pmatrix}\frac{w(x)}{2\pi i}\\
				&=\frac{w(x)}{2\pi i(x-y)}\left(\Psi^\theta_{|\mathbf v}(y)^{-1}
				\Psi^\theta_{|\mathbf v}(x)\right)_{21}.
			\end{align}

			Let us illustrate the above procedure in some examples.
			\begin{example}			
	Let $p_k$ be the normalized degree $k$ orthogonal polynomial with respect to a weight function $w$ on $\Lambda=\mathbb R$, with leading coefficient $\kappa_k>0$. Write
\begin{equation}
\Psi(z) := 
\begin{pmatrix}
\frac{1}{\kappa_N}p_N(z) & \frac{1}{2\pi i \kappa_N}\int_{-\infty}^{+\infty}\frac{p_N(s)w(s)ds}{s-z}\\
-2\pi i\kappa_{N-1} p_{N-1}(z) & -\kappa_{N-1}\int_{-\infty}^{+\infty}\frac{p_{N-1}(s)w(s)ds}{s-z}
\end{pmatrix}.\end{equation}
This is the solution of the Fokas-Its-Kitaev RH problem \cite{FokasItsKitaev}, which is the above RH problem for $\Psi$ with
	\[\Lambda=\mathbb R,\qquad \Psi_\infty(z)=z^{N\sigma_3},\qquad \sigma_3:=\begin{pmatrix}1&0\\0&-1\end{pmatrix}.\]
With $\mathbf f, \mathbf g$ as in			\eqref{fg Psi}, the kernel $K_N(x,y)$ is then the Christoffel-Darboux kernel (note the factor $w(x)$ which was not present in \eqref{eq:CD}; this is due to the different reference measures $\dx x$ here and $w(x)\dx x$ in \eqref{eq:CD})
\[K_N(x,y)=\frac{\kappa_{N-1}w(x)}{\kappa_N}\frac{p_N(x)p_{N-1}(y)-p_N(y)p_{N-1}(x)}{x-y}.\]
The RH problem for $\Psi^\theta_{|\emptyset}$ is then the Fokas-Its-Kitaev RH problem, but with a deformed weight function $(1-\theta)w$; for non-empty $\mathbf v$, the function {$\Psi^\theta_{|\mathbf v}(z)\prod_{v\in\mathbf v}(z-v)^{-\sigma_3}$} then satisfies the Fokas-Its-Kitaev RH problem with weight function $(1-\theta(z))w(z)\prod_{v\in\mathbf v}(z-v)^2$ and with $N$ replaced by $N$ minus the cardinality of $\mathbf v$, which is in perfect agreement with Proposition \ref{prop:OPE}. This RH problem has been an object of intensive study in the past decades and large $N$ asymptotics for its solution have been obtained for a large class of weight functions, see e.g.\ \cite{Deift, KuijlaarsMOP, Kuijlaars}.
\end{example}		
\begin{example}			
Write
\[\Psi(z)=\begin{cases}
	\begin{pmatrix}e^{\pi iz} & e^{\pi iz}\\ -e^{-\pi iz} & 0\end{pmatrix}, & \mbox{for \( \Im z>0\)}, \\
	\begin{pmatrix}e^{\pi iz} & 0 \\ -e^{-\pi iz} & e^{-\pi iz}\end{pmatrix}, & \mbox{for \( \Im z<0\)}.
\end{cases}\]
			This matrix
			satisfies the RH problem for $\Psi$ with 
			\[\Lambda=\mathbb R,\qquad w(x)=1,\qquad\Psi_\infty=\Psi.
	\]
With $\mathbf f, \mathbf g$ as in			\eqref{fg Psi}, the kernel $K(x,y)$ is then the sine kernel \eqref{sine kernel}. The associated RH problem for $\Psi^\theta_{|\emptyset}$ for $\theta=1_B$ the indicator function of a union of intervals was the  RH problem studied originally in \cite{DIZ}, and was also analyzed succesfully in  \cite{BothnerDeiftItsKrasovsky} for $\theta=\gamma 1_B$ with $\gamma\in(0,1)$.
\end{example}				
			\begin{example}			
Write
\begin{equation}
\Psi(z) := \sqrt{2 \pi} e^{\frac{\pi i}{6}} \times \begin{cases}
\begin{pmatrix}
{\rm Ai}(z) & {\rm Ai}(\omega^{2}z) \\
-i{\rm Ai}^{\prime}(z) & -i\omega^{2}{\rm Ai}^{\prime}(\omega^{2}z)
\end{pmatrix}e^{-\frac{\pi i}{6}\sigma_{3}}, & \mbox{for $\Im z>0$}, \\
\begin{pmatrix}
{\rm Ai}(z) & - \omega^{2}{\rm Ai}(\omega z) \\
-i{\rm Ai}^{\prime}(z) & i{\rm Ai}^{\prime}(\omega z)
\end{pmatrix}e^{-\frac{\pi i}{6}\sigma_{3}}, & \mbox{for }\Im z<0,
\end{cases}
\end{equation}
with $\omega = e^{\frac{2\pi i}{3}}$ and ${\rm Ai}$ the Airy function.
Using the relation ${\rm Ai}(z)+\omega{\rm Ai}(\omega z)+\omega^2{\rm Ai}(\omega^2 z)=0$, one verifies, using the asymptotic behavior of the Airy function, that this matrix
			satisfies the RH problem for $\Psi$ with 
			\[\Lambda=\mathbb R,\quad w(x)=1,\quad \Psi_\infty(z)=\begin{cases}\frac{1}{\sqrt 2}z^{-\frac{\sigma_{3}}{4}}\begin{pmatrix}1&i\\i&1\end{pmatrix}  e^{-\frac{2}{3}z^{3/2}\sigma_{3}}&\mbox{for $|\arg z|<\pi-\delta$},\\
			\frac{1}{\sqrt 2}z^{-\frac{\sigma_{3}}{4}}\begin{pmatrix}1&i\\i&1\end{pmatrix}  e^{-\frac{2}{3}z^{3/2}\sigma_{3}}\begin{pmatrix}
1&0\\\pm 1&1			\end{pmatrix}
&\mbox{for $|\arg z|<\pi-\delta$, $\pm\Im z>0$,}			\end{cases}
			\]
			for any sufficiently small $\delta>0$,
with principal branches of the root functions.
With $\mathbf f, \mathbf g$ as in			\eqref{fg Psi}, the kernel $K(x,y)$ is then the Airy kernel \[K(x,y)=\frac{{\rm Ai}(x){\rm Ai}'(y)-{\rm Ai}(y){\rm Ai}'(x)}{x-y}.\]
The RH problem for $\Psi^\theta_{|\emptyset}$, or an equivalent RH problem obtained after opening of the lenses, was then studied in \cite{CharlierClaeys2, CafassoClaeys, CafassoClaeysRuzza} for a rather large class of functions $\theta$, in order to derive differential equations and asymptotics for Airy kernel Fredholm determinants of the form
$\det(1-\mathrm M_{\sqrt\theta}\mathrm K\mathrm M_{\sqrt\theta})$. In particular, these determinants are important in the study of the narrow wedge solution of the Kardar-Parisi-Zhang equation and in the study of finite temperature free fermions, and they have a remarkably rich integrable structure: they are connected to the Korteweg-de Vries equation and to an integro-differential version of the second Painlev\'e equation. The asymptotics resulting from this RH analysis allow also to derive asymptotics for the conditional kernels $K^{\theta}_{|\emptyset}$.
Moreover, the density of the pushed Coulomb gas from \cite{CGKLDT, KrajenbrinkLeDoussal} can be interpreted as an approximation of the one-point function $K^{\theta}_{|\emptyset}(x,x)$.
\end{example}

\medskip

The conclusion of this section is two-fold. First, we just showed that the IIKS RH problem allows one to characterize the conditional kernels $K^\theta_{|\emptyset}$ and $K^\theta_{|\mathbf v}$ in terms of a RH problem, which can potentially be analyzed asymptotically. Secondly, the conditional ensembles $\mathbb P^\theta_{|\emptyset}$ enable us to give a natural probabilistic interpretation to the IIKS method, as we explain next.

\medskip

The starting point of the IIKS method to study Fredholm determinants of the form $\det(1-\mathrm M_\theta\mathrm K)$, is the 			
			Jacobi identity: if $\theta(x)=\theta_t(x)$ depends smoothly on a deformation parameter $t$, we have	
						\begin{align*}\partial_t \log \det(1-\mathrm M_{\sqrt{\theta_t}} \mathrm K\mathrm M_{\sqrt{\theta_t}})&=-\Tr \left[\partial_t\mathrm M_{\theta_{t}} \mathrm K(1-\mathrm M_{\theta_t} \mathrm K)^{-1}\right]
				=-\int_{\Lambda}\partial_t\theta_t(x)K^{\theta_{t}}_{|\emptyset}(x,x)\dx\mu(x).
				\end{align*}
		In analytic terms, this implies that 	
one can compute the Fredholm determinant
$\det(1-\mathrm M_{\theta_t}\mathrm K)$, or at least its logarithmic derivative, provided that one has sufficiently accurate knowledge of the conditional kernel $K^{\theta_{t}}_{|\emptyset}(x,x)$. 
				
				\medskip

In probabilistic terms, if $1-\theta_t$ does not vanish, this identity reads
			\begin{equation}
				\partial_t\log \mathbb E \prod_{x\in{\rm supp}\,\xi}(1-\theta_t(x))=\mathbb E^{\theta_t}_{|\emptyset}\int_\Lambda\partial_t\log(1-\theta_t(x))\dx\xi(x).
			\end{equation}
The logarithmic derivative of an average multiplicative statistic in $\mathbb P$ is thus equal to an average linear statistic in the conditional ensemble $\mathbb P^\theta_{|\emptyset}$.			
			Moreover, if the function $t\mapsto \theta_t(x)$ is a smooth probability distribution function, then the function
		\begin{equation} 
				h^\theta_x(t)=-\partial_t\log(1-\theta_t(x))=\frac{\partial_t\theta_t(x)}{1-\theta_t(x)}=\mbox{Prob}\left(t_x= t\ |\ t_x\geq t\right),
			\end{equation}
has the natural interpretation of a hazard rate likelihood of the random variable $t_x$ with distribution $t\mapsto \theta_t(x)$. We can interpret $t_x$ for instance as the detection time of point $x$, and then $h_x^\theta(t)$ is the likelihood to detect the particle at position $x\in{\rm supp}\,\xi$ at time $t$, given that it was not detected before time $t$.

	\subsubsection*{Acknowledgements}
	The authors were supported by the Fonds de la Recherche Scientifique-FNRS under EOS project O013018F. They are grateful to Alexander Bufetov for instructive discussions about rigidity, and to Guilherme Silva for useful discussions.

			\medskip \ \medskip

			%\noindent
			%{\small{\sc AMS Subject Classification (2010)}: 41A60, 60B20, 33B15, 33E20, 35Q15.}
			%
			%\noindent
			
			%

		\end{document}